\documentclass[11pt]{amsart}%
\usepackage{amsmath}
\usepackage{graphicx}
\usepackage{amsfonts}
\usepackage{amssymb}%
\newtheorem{theorem}{Theorem}[section]
\theoremstyle{plain}

\newtheorem{lemma}{Lemma}[section]

\newtheorem{proposition}{Proposition}[section]

\numberwithin{equation}{section}
\begin{document}
\title[Hessian Estimates]{Hessian and gradient estimates for three dimensional special Lagrangian
Equations with large phase}
\author{Micah Warren}
\author{Yu YUAN}
\address{Department of Mathematics, Box 354350\\
University of Washington\\
Seattle, WA 98195}
\email{mwarren@math.washington.edu, yuan@math.washington.edu}
\thanks{Y.Y. is partially supported by an NSF grant.}
\date{\today}

\begin{abstract}
We derive a priori interior Hessian and gradient estimates for special
Lagrangian equation of phase at least a critical value in dimension three.

\end{abstract}
\maketitle

\section{\bigskip Introduction}

In this paper, we establish a priori\emph{ interior} Hessian and gradient
estimates for the special Lagrangian equation%

\begin{equation}
\sum_{i=1}^{n}\arctan\lambda_{i}=\Theta\label{EsLag}%
\end{equation}
with (the phase) $\left\vert \Theta\right\vert \geq\pi/2$ and $n=3,$ where
$\lambda_{i}$ are the eigenvalues of the Hessian $D^{2}u.$

Equation (\ref{EsLag}) is from the special Lagrangian geometry [HL]. The
Lagrangian graph $\left(  x,Du\left(  x\right)  \right)  \subset\mathbb{R}%
^{n}\times\mathbb{R}^{n}$ is called special when the phase or the argument of
the complex number $\left(  1+\sqrt{-1}\lambda_{1}\right)  \cdots\left(
1+\sqrt{-1}\lambda_{n}\right)  $ is constant $\Theta,$ that is, $u$ satisfies
equation (\ref{EsLag}), and it is special if and only if $\left(  x,Du\left(
x\right)  \right)  $ is a (volume minimizing) minimal surface in
$\mathbb{R}^{n}\times\mathbb{R}^{n}$ [HL, Theorem 2.3, Proposition 2.17]. Note
that equation (\ref{EsLag}) with $n=3$ and $\left\vert \Theta\right\vert
=\pi/2$ or $\pi$ also takes the following forms respectively%
\[
\sigma_{2}\left(  D^{2}u\right)  =\lambda_{1}\lambda_{2}+\lambda_{2}%
\lambda_{3}+\lambda_{3}\lambda_{1}=1
\]
or%
\begin{equation}
\bigtriangleup u=\det D^{2}u. \label{s1=s3}%
\end{equation}

We first state the following interior Hessian estimates.

\begin{theorem}
Let $u$ be a smooth solution to (\ref{EsLag}) with $\left\vert \Theta
\right\vert \geq\pi/2$ and $n=3$ on $B_{R}(0)\subset\mathbb{R}^{3}.$ Then we
have
\[
|D^{2}u(0)|\leq C(3)\exp\left[  C(3)\left(  \cot\frac{\left\vert
\Theta\right\vert -\pi/2}{3}\right)  ^{2}\max_{B_{R}(0)}|Du|^{7}/R^{7}\right]
,
\]
and also
\[
|D^{2}u(0)|\leq C(3)\exp\left\{  C(3)\exp\left[  C(3)\max_{B_{R}(0)}%
|Du|^{3}/R^{3}\right]  \right\}  .
\]

\end{theorem}

The above $\Theta$-independent Hessian estimates make use of the following
Hessian estimate for the three dimensional special Lagrangian equation
(\ref{EsLag}) with the critical phase $\left\vert \Theta\right\vert =\pi/2.$

\begin{theorem}
[{[WY2]}]Let $u$ be a smooth solution to (\ref{EsLag}) with $\left\vert
\Theta\right\vert =\pi/2$ and $n=3$ on $B_{R}(0)\subset\mathbb{R}^{3}.$ Then
we have
\[
|D^{2}u(0)|\leq C(3)\exp\left[  C(3)\max_{B_{R}(0)}|Du|^{3}/R^{3}\right]  .
\]

\end{theorem}

In order to link the dependence of Hessian estimates in the above theorems to
the potential $u$ itself, we have the following gradient estimate in general dimensions.

\begin{theorem}
Let $u$ be a smooth solution to (\ref{EsLag}) with $\left\vert \Theta
\right\vert \geq$ $\left(  n-2\right)  \frac{\pi}{2}$ on $B_{3R}%
(0)\subset\mathbb{R}^{n}.$ Then we have
\begin{equation}
\max_{B_{R}(0)}|Du|\leq C\left(  n\right)  \left[  \operatorname*{osc}%
_{B_{3R}\left(  0\right)  }\frac{u}{R}+1\right]  . \label{gradient}%
\end{equation}
\ \ 
\end{theorem}

One quick consequence of the above estimates is a Liouville type result for
global solutions with quadratic growth to (\ref{EsLag}) with $\left\vert
\Theta\right\vert =\pi/2$ and $n=3,$ namely any such a solution must be
quadratic (cf. [Y1], [Y2] where other Liouville type results for convex
solutions to (\ref{EsLag}) and Bernstein type results for global solutions to
(\ref{EsLag}) with $\left\vert \Theta\right\vert >\left(  n-2\right)  \pi/2$
were obtained). Another application is the regularity (analyticity) of the
$C^{0}$ viscosity solutions to (\ref{EsLag}) with $\left\vert \Theta
\right\vert \geq\pi/2$ and $n=3.$

In the 1950's, Heinz [H] derived a Hessian bound for the two dimensional
Monge-Amp\`{e}re type equation including (\ref{EsLag}) with $n=2;$ see also
Pogorelov [P1] for Hessian estimates for these equations including
(\ref{EsLag}) with $\left\vert \Theta\right\vert >\pi/2\ $and $n=2.$ In the
1970's Pogorelov [P2] constructed his famous counterexamples, namely irregular
solutions to three dimensional Monge-Amp\`{e}re equations $\sigma_{3}%
(D^{2}u)=\lambda_{1}\lambda_{2}\lambda_{3}=\det(D^{2}u)=1;$ see
generalizations of the counterexamples for $\sigma_{k}$ equations with
$k\geq3$ in Urbas [U1]. In passing, we also mention Hessian estimates for
solutions with certain strict convexity constraints to Monge-Amp\`{e}re
equations and $\sigma_{k}$ equation ($k\geq2$) by Pogorelov [P2] and Chou-Wang
[CW] respectively using the Pogorelov technique. Urbas [U2][U3], also Bao and
Chen [BC] obtained (pointwise) Hessian estimates in term of certain integrals
of the Hessian, for $\sigma_{k}$ equations and special Lagrangian equation
(1.1) with $n=3,\ \Theta=\pi$ respectively.

A Hessian bound for (\ref{EsLag}) with $n=2$ also follows from an earlier work
by Gregori [G], where Heinz's Jacobian estimate was extended to get a gradient
bound in terms of the heights of the two dimensional minimal surfaces with any
codimension. A gradient estimate for general dimensional and codimensional
minimal graphs with certain constraints on the gradients themselves was
obtained in [W], using an integral method developed for codimension one
minimal graphs. The gradient estimate of Bombieri-De Giorgi-Miranda [BDM] (see
also [T1] [BG] [K]) is by now classic.

The Bernstein-Pogorelov-Korevaar technique was employed to derive Hessian
estimates for (\ref{EsLag}) with certain constraints on the solutions in
[WY1].\ A slightly sharper Hessian estimate for (\ref{EsLag}) with $n=2$ was
obtained by elementary methods in [WY3]. The Hessian estimate for the equation
$\sigma_{2}\left(  D^{2}u\right)  =1$ in dimension three, or (\ref{EsLag})
with $\left\vert \Theta\right\vert =\pi/2$ and $n=3$ was derived by
\textquotedblleft less\textquotedblright\ involved arguments in [WY2].

The heuristic ideas for Hessian estimates are as follows. The function
$b=\ln\sqrt{1+\lambda_{\max}^{2}}$ is subharmonic so that $b$ at any point is
bounded by its integral over a ball around the point on the minimal surface by
Michael-Simon's mean value inequality [MS]. This special choice of $b$ is not
only subharmonic, but even stronger, satisfies a Jacobi inequality. Coupled
with Sobolev inequalities for functions both with and without compact support,
this Jacobi inequality leads to a bound on the integral of $b$ by the volume
of the ball on the minimal surface. Taking advantage of the divergence form of
the volume element of the minimal Lagrangian graph, we bound the volume in
terms of the height of the special Lagrangian graph, which is the gradient of
the solution to equation (\ref{EsLag}).

As for the gradient estimates, we adapt Trudinger's method [T2] for
$\sigma_{k}$ equations to (\ref{EsLag}) with the critical phase $\Theta
=\left(  n-2\right)  \pi/2.$ Gradient estimates for (\ref{EsLag}) with larger
phase $\Theta>\left(  n-2\right)  \pi/2$ are straightforward consequences of
the observation that the Hessians of solutions have lower bound depending on
the phase $\Theta.$ In order to obtain the uniform gradient estimates
independent of the phase $\Theta,$ we make use of the Lewy rotation, which
links the corresponding estimates to the ones in the case of the critical phase.

Lewy rotation is also used along with a relative isoperimetric inequality to
get another key ingredient in our proof of the Hessian estimates, namely a
Sobolev inequality for functions without compact support, in the super
critical phase case.

As one can see, our arguments for the Hessian estimates resemble the
\textquotedblleft isoperimetric\textquotedblright\ proof of the classical
gradient estimate for minimal graphs. Now only some technical obstacles remain
for Hessian estimates for (\ref{EsLag}) with large phase $\left\vert
\Theta\right\vert \geq\left(  n-2\right)  \pi/2$ and $n\geq4.$ Yet further new
ideas are lacking for us to handle the special Lagrangian equation
(\ref{EsLag}) with general phases in dimension three and higher, including
(\ref{s1=s3}) corresponding to $\Theta=0$ and $n=3.$

\textbf{Notation. }$\partial_{ij}=\frac{\partial^{2}}{\partial x_{i}\partial
x_{j}},u_{i}=\partial_{ij}u,$ etc., but $\lambda_{1},\cdots,\lambda_{n}$ and
$b_{1}=\ln\sqrt{1+\lambda_{1}^{2}}$, $b_{2}=\left(  \ln\sqrt{1+\lambda_{1}%
^{2}}+\ln\sqrt{1+\lambda_{2}^{2}}\right)  /2\ $do not represent the partial
derivatives. The eigenvalues are ordered $\lambda_{1}\geq\cdots\geq\lambda
_{n}$ and $\theta_{i}=\arctan\lambda_{i}.$ Further, $h_{ijk}$ will denote (the
second fundamental form)
\[
h_{ijk}=\frac{1}{\sqrt{1+\lambda_{i}^{2}}}\frac{1}{\sqrt{1+\lambda_{j}^{2}}%
}\frac{1}{\sqrt{1+\lambda_{k}^{2}}}u_{ijk}.
\]
when $D^{2}u$ is diagonalized. Finally $C\left(  n\right)  $ will denote
various constants depending only on dimension $n.$

\section{Preliminaries}

Taking the gradient of both sides of the special Lagrangian equation
(\ref{EsLag}), we have
\begin{equation}
\sum_{i,j}^{n}g^{ij}\partial_{ij}\left(  x,Du\left(  x\right)  \right)  =0,
\label{Emin}%
\end{equation}
where $\left(  g^{ij}\right)  $ is the inverse of the induced metric
$g=\left(  g_{ij}\right)  =I+D^{2}uD^{2}u$ on the surface $\left(  x,Du\left(
x\right)  \right)  \subset\mathbb{R}^{n}\times\mathbb{R}^{n}.$ Simple
geometric manipulation of (\ref{Emin}) yields the divergence form of the
minimal surface equation
\[
\bigtriangleup_{g}\left(  x,Du\left(  x\right)  \right)  =0,
\]
where the Laplace-Beltrami operator of the metric $g$ is given by
\[
\bigtriangleup_{g}=\frac{1}{\sqrt{\det g}}\sum_{i,j}^{n}\partial_{i}\left(
\sqrt{\det g}g^{ij}\partial_{j}\right)  .
\]
Because we are using harmonic coordinates $\bigtriangleup_{g}x=0,$ we see that
$\bigtriangleup_{g}$ also equals the linearized operator of the special
Lagrangian equation (\ref{EsLag}) at $u,$%
\[
\bigtriangleup_{g}=\sum_{i,j}^{n}g^{ij}\partial_{ij}.
\]
The gradient and inner product with respect to the metric $g$ are
\begin{align*}
\nabla_{g}v  &  =\left(  \sum_{k=1}^{n}g^{1k}v_{k},\cdots,\sum_{k=1}^{n}%
g^{nk}v_{k}\right)  ,\\
\left\langle \nabla_{g}v,\nabla_{g}w\right\rangle _{g}  &  =\sum_{i,j=1}%
^{n}g^{ij}v_{i}w_{j},\ \ \text{in particular \ }\left\vert \nabla
_{g}v\right\vert ^{2}=\left\langle \nabla_{g}v,\nabla_{g}v\right\rangle _{g}.
\end{align*}

\subsection{Jacobi inequality}

We begin with some geometric calculations.

\begin{lemma}
Let $u$ be a smooth solution to (\ref{EsLag}). Suppose that the Hessian
$D^{2}u$ is diagonalized and the eigenvalue $\lambda_{1}$ is distinct from all
other eigenvalues of $D^{2}u$ at point $p.$ Set $b_{1}=\ln\sqrt{1+\lambda
_{1}^{2}}$ near $p.$ \ Then we have at $p$%
\begin{equation}
\left\vert \nabla_{g}b_{1}\right\vert ^{2}=\sum_{k=1}^{n}\lambda_{1}%
^{2}h_{11k}^{2} \label{gradientb1}%
\end{equation}
and
\begin{gather}
\bigtriangleup_{g}b_{1}=\nonumber\\
(1+\lambda_{1}^{2})h_{111}^{2}+\sum_{k>1}\left(  \frac{2\lambda_{1}}%
{\lambda_{1}-\lambda_{k}}+\frac{2\lambda_{1}^{2}\lambda_{k}}{\lambda
_{1}-\lambda_{k}}\right)  h_{kk1}^{2}\label{lapb1a}\\
+\sum_{k>1}\left[  1+\frac{2\lambda_{1}}{\lambda_{1}-\lambda_{k}}%
+\frac{\lambda_{1}^{2}\left(  \lambda_{1}+\lambda_{k}\right)  }{\lambda
_{1}-\lambda_{k}}\right]  h_{11k}^{2}\label{lapb1b}\\
+\sum_{k>j>1}2\lambda_{1}\left[  \frac{1+\lambda_{k}^{2}}{\lambda_{1}%
-\lambda_{k}}+\frac{1+\lambda_{j}^{2}}{\lambda_{1}-\lambda_{j}}+(\lambda
_{j}+\lambda_{k})\right]  h_{kj1}^{2}. \label{lapb1c}%
\end{gather}

\end{lemma}

\begin{proof}
We first compute the derivatives of the smooth function $b_{1}$ near $p.$ We
may implicitly differentiate the characteristic equation
\[
\det(D^{2}u-\lambda_{1}I)=0
\]
near any point where $\lambda_{1}$ is distinct from the other eigenvalues.
Then we get at $p$
\begin{align*}
\partial_{e}\lambda_{1}  &  =\partial_{e}u_{11},\\
\partial_{ee}\lambda_{1}  &  =\partial_{ee}u_{11}+\sum_{k>1}2\frac{\left(
\partial_{e}u_{1k}\right)  ^{2}}{\lambda_{1}-\lambda_{k}},
\end{align*}
with arbitrary unit vector $e\in\mathbb{R}^{n}.$

Thus we have (\ref{gradientb1}) at $p$%
\[
|\nabla_{g}b_{1}|^{2}=\sum_{k=1}^{n}g^{kk}\left(  \frac{\lambda_{1}}%
{1+\lambda_{1}^{2}}\partial_{k}u_{11}\right)  ^{2}=\sum_{k=1}^{n}\lambda
_{1}^{2}h_{11k}^{2},
\]
where we used the notation $h_{ijk}=\sqrt{g^{ii}}\sqrt{g^{jj}}\sqrt{g^{kk}%
}u_{ijk}.$

From
\[
\partial_{ee}b_{1}=\partial_{ee}\ln\sqrt{1+\lambda_{1}^{2}}=\frac{\lambda_{1}%
}{1+\lambda_{1}^{2}}\partial_{ee}\lambda_{1}+\frac{1-\lambda_{1}^{2}}{\left(
1+\lambda_{1}^{2}\right)  ^{2}}\left(  \partial_{e}\lambda_{1}\right)  ^{2},
\]
we conclude that at $p$
\[
\partial_{ee}b_{1}=\frac{\lambda_{1}}{1+\lambda_{1}^{2}}\left[  \partial
_{ee}u_{11}+\sum_{k>1}2\frac{\left(  \partial_{e}u_{1k}\right)  ^{2}}%
{\lambda_{1}-\lambda_{k}}\right]  +\frac{1-\lambda_{1}^{2}}{\left(
1+\lambda_{1}^{2}\right)  ^{2}}\left(  \partial_{e}u_{11}\right)  ^{2},
\]
and
\begin{align}
\bigtriangleup_{g}b_{1}  &  =\sum_{\gamma=1}^{n}g^{\gamma\gamma}%
\partial_{\gamma\gamma}b_{1}\nonumber\\
&  =\sum_{\gamma=1}^{n}g^{\gamma\gamma}\frac{\lambda_{1}}{1+\lambda_{1}^{2}%
}\left(  \partial_{\gamma\gamma}u_{11}+\sum_{k>1}2\frac{\left(  u_{1k\gamma
}\right)  ^{2}}{\lambda_{1}-\lambda_{k}}\right)  +\sum_{\gamma=1}^{n}%
\frac{1-\lambda_{1}^{2}}{\left(  1+\lambda_{1}^{2}\right)  ^{2}}%
g^{\gamma\gamma}u_{11\gamma}^{2}. \label{E4thorder}%
\end{align}

Next we substitute the fourth order derivative terms $\partial_{\gamma\gamma
}u_{11}$ in the above by lower order derivative terms.\ Differentiating the
minimal surface equation (\ref{Emin}) $\sum_{\alpha,\beta=1}^{n}g^{\alpha
\beta}u_{j\alpha\beta}=0,$\ we obtain
\begin{align}
\bigtriangleup_{g}u_{ij}  &  =\sum_{\alpha,\beta=1}^{n}g^{\alpha\beta
}u_{ji\alpha\beta}=\sum_{\alpha,\beta=1}^{n}-\partial_{i}g^{_{_{\alpha\beta}}%
}u_{j\alpha\beta}=\sum_{\alpha,\beta,\gamma,\delta=1}^{n}g^{\alpha\gamma
}\partial_{i}g_{\gamma\delta}g^{\delta\beta}u_{j\alpha\beta}\nonumber\\
&  =\sum_{\alpha,\beta=1}^{n}g^{\alpha\alpha}g^{\beta\beta}(\lambda_{\alpha
}+\lambda_{\beta})u_{\alpha\beta i}u_{\alpha\beta j}, \label{lapujj}%
\end{align}
where we used
\[
\partial_{i}g_{_{\gamma\delta}}=\partial_{i}(\delta_{\gamma\delta}%
+\sum_{\varepsilon=1}^{n}u_{\gamma\varepsilon}u_{\varepsilon\delta}%
)=u_{\gamma\delta i}(\lambda_{\gamma}+\lambda_{\delta})
\]
with diagonalized $D^{2}u.$ Plugging (\ref{lapujj}) with $i=j=1$ in
(\ref{E4thorder}), we have at $p$%

\begin{align*}
\bigtriangleup_{g}b_{1}  &  =\frac{\lambda_{1}}{1+\lambda_{1}^{2}}\left[
\sum_{\alpha,\beta=1}^{n}g^{\alpha\alpha}g^{\beta\beta}(\lambda_{\alpha
}+\lambda_{\beta})u_{\alpha\beta1}^{2}+\sum_{\gamma=1}^{n}\sum_{k>1}%
2\frac{u_{1k\gamma}^{2}}{\lambda_{1}-\lambda_{k}}g^{\gamma\gamma}\right] \\
&  +\sum_{\gamma=1}^{n}\frac{1-\lambda_{1}^{2}}{\left(  1+\lambda_{1}%
^{2}\right)  ^{2}}g^{\gamma\gamma}u_{11\gamma}^{2}\\
&  =\lambda_{1}\sum_{\alpha,\beta=1}^{n}(\lambda_{\alpha}+\lambda_{\beta
})h_{\alpha\beta1}^{2}+\sum_{k>1}\sum_{\gamma=1}^{n}\frac{2\lambda_{1}\left(
1+\lambda_{k}^{2}\right)  }{\lambda_{1}-\lambda_{k}}h_{\gamma k1}^{2}%
+\sum_{\gamma=1}^{n}\left(  1-\lambda_{1}^{2}\right)  h_{11\gamma}^{2},
\end{align*}
where we used the notation $h_{ijk}=\sqrt{g^{ii}}\sqrt{g^{jj}}\sqrt{g^{kk}%
}u_{ijk}.$ Regrouping those terms $h_{\heartsuit\heartsuit1},\ h_{11\heartsuit
},$ and $h_{\heartsuit\clubsuit1}$ in the last expression, we have
\begin{gather*}
\bigtriangleup_{g}b_{1}=\left(  1-\lambda_{1}^{2}\right)  h_{111}^{2}%
+\sum_{\alpha=1}^{n}2\lambda_{1}\lambda_{\alpha}h_{\alpha\alpha1}^{2}%
+\sum_{k>1}\frac{2\lambda_{1}\left(  1+\lambda_{k}^{2}\right)  }{\lambda
_{1}-\lambda_{k}}h_{kk1}^{2}\\
+\sum_{k>1}2\lambda_{1}(\lambda_{k}+\lambda_{1})h_{k11}^{2}+\sum_{k>1}\left(
1-\lambda_{1}^{2}\right)  h_{11k}^{2}+\sum_{k>1}\frac{2\lambda_{1}\left(
1+\lambda_{k}^{2}\right)  }{\lambda_{1}-\lambda_{k}}h_{1k1}^{2}\\
+\sum_{k>j>1}2\lambda_{1}(\lambda_{j}+\lambda_{k})h_{jk1}^{2}+\sum
_{\substack{j,k>1,\\j\neq k}}\frac{2\lambda_{1}\left(  1+\lambda_{k}%
^{2}\right)  }{\lambda_{1}-\lambda_{k}}h_{jk1}^{2}.
\end{gather*}
After simplifying the above expression, we have the second formula in Lemma 2.1.
\end{proof}

\begin{lemma}
Let $u$ be a smooth solution to (\ref{EsLag}) with $n=3$ and $\Theta\geq
\pi/2.$ Suppose that the ordered eigenvalues $\lambda_{1}\geq\lambda_{2}%
\geq\lambda_{3}$ of the Hessian $D^{2}u$ satisfy $\lambda_{1}>\lambda_{2}$ at
point $p.$ Set
\[
b_{1}=\ln\sqrt{1+\lambda_{\max}^{2}}=\ln\sqrt{1+\lambda_{1}^{2}}%
\]
Then we have at $p$%
\begin{equation}
\bigtriangleup_{g}b_{1}\geq\frac{1}{3}|\nabla_{g}b_{1}|^{2}. \label{Jacobi-3d}%
\end{equation}

\end{lemma}

\begin{proof}
[Proof ]We assume that the Hessian $D^{2}u$ is diagonalized at point $p.$

Step 1. Recall $\theta_{i}=\arctan\lambda_{i}\in(-\pi/2,\pi/2)$ and
\ $\theta_{1}+\theta_{2}+\theta_{3}=\Theta\geq\pi/2.$ It is easy to see that
$\theta_{1}\geq\theta_{2}>0$ and $\theta_{i}+\theta_{j}\geq0$ for any pair.
Consequently $\lambda_{1}\geq\lambda_{2}>0\ $and$\ \lambda_{i}+\lambda_{j}%
\geq0$ for any pair of distinct eigenvalues. It follows that (\ref{lapb1c}) in
the formula for $\bigtriangleup_{g}b_{1}$ is positive, then from
(\ref{lapb1a}) and (\ref{lapb1b}) we have the inequality\
\begin{equation}
\bigtriangleup_{g}b_{1}\geq\lambda_{1}^{2}\left(  h_{111}^{2}+\sum_{k>1}%
\frac{2\lambda_{k}}{\lambda_{1}-\lambda_{k}}h_{kk1}^{2}\right)  +\lambda
_{1}^{2}\sum_{k>1}\left(  1+\frac{2\lambda_{k}}{\lambda_{1}-\lambda_{k}%
}\right)  h_{11k}^{2}. \label{cor21z}%
\end{equation}

Combining (\ref{cor21z}) and (\ref{gradientb1}) gives \
\begin{gather}
\bigtriangleup_{g}b_{1}-\frac{1}{3}|\nabla_{g}b_{1}|^{2}\geq\nonumber\\
\lambda_{1}^{2}\left(  \frac{2}{3}h_{111}^{2}+\sum_{k>1}\frac{2\lambda_{k}%
}{\lambda_{1}-\lambda_{k}}h_{kk1}^{2}\right)  +\lambda_{1}^{2}\sum_{k>1}%
\frac{2\left(  \lambda_{1}+2\lambda_{k}\right)  }{3\left(  \lambda_{1}%
-\lambda_{k}\right)  }h_{11k}^{2}. \label{cor21a}%
\end{gather}

Step 2. \ We show that the last term in (\ref{cor21a}) is nonnegative. Note
that $\lambda_{1}+2\lambda_{k}\geq\lambda_{1}+2\lambda_{3}.$ We only need to
show that $\lambda_{1}+2\lambda_{3}\geq0$ in the case that $\lambda_{3}<0$ or
equivalently $\theta_{3}<0.$ From $\theta_{1}+\theta_{2}+\theta_{3}=\Theta
\geq$ $\pi/2,$ we have
\[
\frac{\pi}{2}>\theta_{3}+\frac{\pi}{2}=\left(  \frac{\pi}{2}-\theta
_{1}\right)  +\left(  \frac{\pi}{2}-\theta_{2}\right)  +\Theta-\frac{\pi}%
{2}\geq2\left(  \frac{\pi}{2}-\theta_{1}\right)  .
\]
It follows that
\[
-\frac{1}{\lambda_{3}}=\tan\left(  \theta_{3}+\frac{\pi}{2}\right)
>2\tan\left(  \frac{\pi}{2}-\theta_{1}\right)  =\frac{2}{\lambda_{1}},
\]
then
\begin{equation}
\lambda_{1}+2\lambda_{3}>0. \label{l_1+2l_3>0}%
\end{equation}
\ 

Step 3. \ We show that the first term in (\ref{cor21a}) is nonnegative by
proving
\begin{equation}
\frac{2}{3}h_{111}^{2}+\frac{2\lambda_{2}}{\lambda_{1}-\lambda_{2}}h_{221}%
^{2}+\frac{2\lambda_{3}}{\lambda_{1}-\lambda_{3}}h_{331}^{2}\geq0.
\label{cor21c}%
\end{equation}
We only need to show it for $\lambda_{3}<0.$ Directly from the minimal surface
equation (\ref{Emin})
\[
h_{111}+h_{221}+h_{331}=0,
\]
we bound
\[
h_{331}^{2}=\left(  h_{111}+h_{221}\right)  ^{2}\leq\left(  \frac{2}{3}%
h_{111}^{2}+\frac{2\lambda_{2}}{\lambda_{1}-\lambda_{2}}h_{221}^{2}\right)
\left(  \frac{3}{2}+\frac{\lambda_{1}-\lambda_{2}}{2\lambda_{2}}\right)  .
\]
It follows that
\begin{gather*}
\frac{2}{3}h_{111}^{2}+\frac{2\lambda_{2}}{\lambda_{1}-\lambda_{2}}h_{221}%
^{2}+\frac{2\lambda_{3}}{\lambda_{1}-\lambda_{3}}h_{331}^{2}\geq\\
\left(  \frac{2}{3}h_{111}^{2}+\frac{2\lambda_{2}}{\lambda_{1}-\lambda_{2}%
}h_{221}^{2}\right)  \left[  1+\frac{2\lambda_{3}}{\lambda_{1}-\lambda_{3}%
}\left(  \frac{3}{2}+\frac{\lambda_{1}-\lambda_{2}}{2\lambda_{2}}\right)
\right]  .
\end{gather*}
The last factor becomes
\[
1+\frac{2\lambda_{3}}{\lambda_{1}-\lambda_{3}}\left(  \frac{3}{2}%
+\frac{\lambda_{1}-\lambda_{2}}{2\lambda_{2}}\right)  =\frac{\sigma_{2}%
}{\left(  \lambda_{1}-\lambda_{3}\right)  \lambda_{2}}>0.
\]
The above inequality is from the observation
\[
\operatorname{Re}%
%TCIMACRO{\dprod \limits_{i=1}^{3}}%
%BeginExpansion
{\displaystyle\prod\limits_{i=1}^{3}}
%EndExpansion
\left(  1+\sqrt{-1}\lambda_{i}\right)  =1-\sigma_{2}\leq0
\]
for $3\pi/2>\theta_{1}+\theta_{2}+\theta_{3}=\Theta\geq$ $\pi/2.$ Therefore
(\ref{cor21c}) holds.

We have proved the pointwise Jacobi inequality (\ref{Jacobi-3d}) in Lemma 2.2.
\end{proof}

\begin{lemma}
Let $u$ be a smooth solution to (\ref{EsLag}) with $n=3$ and $\Theta\geq
\pi/2.$ Suppose that the ordered eigenvalues $\lambda_{1}\geq\lambda_{2}%
\geq\lambda_{3}$ of the Hessian $D^{2}u$ satisfy $\lambda_{2}>\lambda_{3}$ at
point $p.$ Set
\[
b_{2}=\frac{1}{2}\left(  \ln\sqrt{1+\lambda_{1}^{2}}+\ln\sqrt{1+\lambda
_{2}^{2}}\right)  .
\]
\ Then $b_{2}$ satisfies at $p$
\begin{equation}
\bigtriangleup_{g}b_{2}\geq0. \label{subharmonicb2}%
\end{equation}
Further, suppose that $\lambda_{1}\equiv$ $\lambda_{2}$ in a neighborhood of
$p.$\ Then $b_{2}$ satisfies at $p$
\begin{equation}
\bigtriangleup_{g}b_{2}\geq\frac{1}{3}\left\vert \nabla_{g}b_{2}\right\vert
^{2}. \label{jacobi-b2}%
\end{equation}

\end{lemma}

\begin{proof}
\ We assume that Hessian $D^{2}u$ is diagonalized at point $p.$ We may use
Lemma 2.1 to obtain expressions for both \ $\bigtriangleup_{g}\ln
\sqrt{1+\lambda_{1}^{2}}$ and $\bigtriangleup_{g}\ln\sqrt{1+\lambda_{2}^{2}},$
whenever the eigenvalues of $D^{2}u$ are distinct. From (\ref{lapb1a}),
(\ref{lapb1b}), and (\ref{lapb1c}),\ we have
\begin{gather}
\bigtriangleup_{g}\ln\sqrt{1+\lambda_{1}^{2}}+\bigtriangleup_{g}\ln
\sqrt{1+\lambda_{2}^{2}}=\label{lapb2}\\
(1+\lambda_{1}^{2})h_{111}^{2}+\sum_{k>1}\frac{2\lambda_{1}(1+\lambda
_{1}\lambda_{k})}{\lambda_{1}-\lambda_{k}}h_{kk1}^{2}+\sum_{k>1}\left[
1+\lambda_{1}^{2}+2\lambda_{1}\left(  \frac{1+\lambda_{1}\lambda_{k}}%
{\lambda_{1}-\lambda_{k}}\right)  \right]  h_{11k}^{2}\nonumber\\
+2\lambda_{1}\left[  \frac{1+\lambda_{3}^{2}}{\lambda_{1}-\lambda_{3}}%
+\frac{1+\lambda_{2}^{2}}{\lambda_{1}-\lambda_{2}}+(\lambda_{3}+\lambda
_{2})\right]  h_{321}^{2}\nonumber\\
+(1+\lambda_{2}^{2})h_{222}^{2}+\sum_{k\neq2}\frac{2\lambda_{2}(1+\lambda
_{2}\lambda_{k})}{\lambda_{2}-\lambda_{k}}h_{kk2}^{2}+\sum_{k\neq2}\left[
1+\lambda_{2}^{2}+2\lambda_{2}\left(  \frac{1+\lambda_{2}\lambda_{k}}%
{\lambda_{2}-\lambda_{k}}\right)  \right]  h_{22k}^{2}\nonumber\\
+2\lambda_{2}\left[  \frac{1+\lambda_{3}^{2}}{\lambda_{2}-\lambda_{3}}%
+\frac{1+\lambda_{1}^{2}}{\lambda_{2}-\lambda_{1}}+(\lambda_{3}+\lambda
_{1})\right]  h_{321}^{2}.\text{ \ }\nonumber
\end{gather}
The function $b_{2}$ is symmetric in $\lambda_{1}$ and $\lambda_{2},$ thus
$b_{2}$ is smooth even when $\lambda_{1}=\lambda_{2},$ provided that
$\lambda_{2}>\lambda_{3}.$ \ We simplify (\ref{lapb2}) to the following, which
holds by continuity wherever $\lambda_{1}\geq\lambda_{2}>\lambda_{3}.$%
\[
2\bigtriangleup_{g}b_{2}=
\]%
\begin{align}
&  (1+\lambda_{1}^{2})h_{111}^{2}+(3+\lambda_{2}^{2}+2\lambda_{1}\lambda
_{2})h_{221}^{2}+\left(  \frac{2\lambda_{1}}{\lambda_{1}-\lambda_{3}}%
+\frac{2\lambda_{1}^{2}\lambda_{3}}{\lambda_{1}-\lambda_{3}}\right)
h_{331}^{2}\label{lapb2a}\\
&  +(3+\lambda_{1}^{2}+2\lambda_{1}\lambda_{2})h_{112}^{2}+(1+\lambda_{2}%
^{2})h_{222}^{2}+\left(  \frac{2\lambda_{2}}{\lambda_{2}-\lambda_{3}}%
+\frac{2\lambda_{2}^{2}\lambda_{3}}{\lambda_{2}-\lambda_{3}}\right)
h_{332}^{2}\label{lapb2b}\\
&  +\left[  \frac{3\lambda_{1}-\lambda_{3}+\lambda_{1}^{2}(\lambda_{1}%
+\lambda_{3})}{\lambda_{1}-\lambda_{3}}\right]  h_{113}^{2}+\left[
\frac{3\lambda_{2}-\lambda_{3}+\lambda_{2}^{2}(\lambda_{2}+\lambda_{3}%
)}{\lambda_{2}-\lambda_{3}}\right]  h_{223}^{2}\label{lapb2c}\\
&  +2\left[  1+\lambda_{1}\lambda_{2}+\lambda_{2}\lambda_{3}+\lambda
_{3}\lambda_{1}+\frac{\lambda_{1}\left(  1+\lambda_{3}^{2}\right)  }%
{\lambda_{1}-\lambda_{3}}+\frac{\lambda_{2}\left(  1+\lambda_{3}^{2}\right)
}{\lambda_{2}-\lambda_{3}}\right]  h_{123}^{2}. \label{lapb2d}%
\end{align}
Using the relations $\lambda_{1}\geq\lambda_{2}>0,\ \lambda_{i}+\lambda
_{j}>0,$ and $\sigma_{2}\geq1$ derived in the proof of Lemma 2.2, we see that
(\ref{lapb2d}) and (\ref{lapb2c}) are nonnegative. We only need to justify the
nonnegativity of (\ref{lapb2a}) and (\ref{lapb2b}) for $\lambda_{3}<0.$ From
the minimal surface equation (\ref{Emin}), we know
\[
h_{332}^{2}=\left(  h_{112}+h_{222}\right)  ^{2}\leq\left[  (\lambda_{1}%
^{2}+2\lambda_{1}\lambda_{2})h_{112}^{2}+\lambda_{2}^{2}h_{222}^{2}\right]
\left(  \frac{1}{\lambda_{1}^{2}+2\lambda_{1}\lambda_{2}}+\frac{1}{\lambda
_{2}^{2}}\right)  .
\]
It follows that
\begin{align*}
(\ref{lapb2b})  &  \geq(\lambda_{1}^{2}+2\lambda_{1}\lambda_{2})h_{112}%
^{2}+\lambda_{2}^{2}h_{222}^{2}+\frac{2\lambda_{2}^{2}\lambda_{3}}{\lambda
_{2}-\lambda_{3}}h_{332}^{2}\\
&  \geq\left[  (\lambda_{1}^{2}+2\lambda_{1}\lambda_{2})h_{112}^{2}%
+\lambda_{2}^{2}h_{222}^{2}\right]  \left[  1+\frac{2\lambda_{2}^{2}%
\lambda_{3}}{\lambda_{2}-\lambda_{3}}\left(  \frac{1}{\lambda_{1}^{2}%
+2\lambda_{1}\lambda_{2}}+\frac{1}{\lambda_{2}^{2}}\right)  \right]  .
\end{align*}
The last term becomes
\begin{align*}
&  \frac{2\lambda_{2}^{2}\lambda_{3}}{\lambda_{2}-\lambda_{3}}\left(
\frac{\lambda_{2}-\lambda_{3}}{2\lambda_{2}^{2}\lambda_{3}}+\frac{1}%
{\lambda_{1}^{2}+2\lambda_{1}\lambda_{2}}+\frac{1}{\lambda_{2}^{2}}\right) \\
&  =\frac{\lambda_{2}}{\lambda_{2}-\lambda_{3}}\left[  \frac{\sigma_{2}%
}{\lambda_{1}\lambda_{2}}-\frac{\lambda_{3}}{\left(  \lambda_{1}+2\lambda
_{2}\right)  }\right]  \geq0.
\end{align*}
Thus (\ref{lapb2b}) is nonnegative. Similarly (\ref{lapb2a}) is nonnegative.
We have proved (\ref{subharmonicb2}).

Next we prove (\ref{jacobi-b2}), still assuming $D^{2}u$ is diagonalized at
point $p.$ Plugging in $\lambda_{1}=\lambda_{2}$ into (\ref{lapb2a}),
(\ref{lapb2b}), and (\ref{lapb2c}), we get
\[
2\bigtriangleup_{g}b_{2}\geq
\]%
\begin{align*}
&  \lambda_{1}^{2}\left(  h_{111}^{2}+3h_{221}^{2}+\frac{2\lambda_{3}}%
{\lambda_{1}-\lambda_{3}}h_{331}^{2}\right) \\
&  +\;\lambda_{1}^{2}\left(  3h_{112}^{2}+h_{222}^{2}+\frac{2\lambda_{3}%
}{\lambda_{1}-\lambda_{3}}h_{332}^{2}\right) \\
&  +\lambda_{1}^{2}\left(  \frac{\lambda_{1}+\lambda_{3}}{\lambda_{1}%
-\lambda_{3}}\right)  \left(  h_{113}^{2}+h_{223}^{2}\right)  .
\end{align*}
Differentiating the eigenvector equations in the neighborhood where
$\lambda_{1}\equiv$ $\lambda_{2}$%
\[
\left(  D^{2}u\right)  U=\frac{\lambda_{1}+\lambda_{2}}{2}U,\ \left(
D^{2}u\right)  V=\frac{\lambda_{1}+\lambda_{2}}{2}V,\ \text{and}\ \left(
D^{2}u\right)  W=\lambda_{3}W,
\]
we see that $u_{11e}=u_{22e}$ for any $e\in\mathbb{R}^{3}$ at point $p.$ Using
the minimal surface equation (\ref{Emin}), we then have
\[
h_{11k}=h_{22k}=-\frac{1}{2}h_{33k}%
\]
at point $p.$ Thus
\[
\bigtriangleup_{g}b_{2}\geq\lambda_{1}^{2}\left[  2\left(  \frac{\lambda
_{1}+\lambda_{3}}{\lambda_{1}-\lambda_{3}}\right)  h_{111}^{2}+2\left(
\frac{\lambda_{1}+\lambda_{3}}{\lambda_{1}-\lambda_{3}}\right)  h_{112}%
^{2}+\left(  \frac{\lambda_{1}+\lambda_{3}}{\lambda_{1}-\lambda_{3}}\right)
h_{113}^{2}\right]  .
\]
The gradient $|\nabla_{g}b_{2}|^{2}$ has the expression at $p$
\[
|\nabla_{g}b_{2}|^{2}=\sum_{k=1}^{3}g^{kk}\left(  \frac{1}{2}\frac{\lambda
_{1}}{1+\lambda_{1}^{2}}\partial_{k}u_{11}+\frac{1}{2}\frac{\lambda_{2}%
}{1+\lambda_{2}^{2}}\partial_{k}u_{22}\right)  ^{2}=\sum_{k=1}^{3}\lambda
_{1}^{2}h_{11k}^{2}.
\]
Thus at $p$%
\begin{gather*}
\bigtriangleup_{g}b_{2}-\frac{1}{3}|\nabla_{g}b_{2}|^{2}\geq\\
\lambda_{1}^{2}\left\{  \left[  2\left(  \frac{\lambda_{1}+\lambda_{3}%
}{\lambda_{1}-\lambda_{3}}\right)  -\frac{1}{3}\right]  h_{111}^{2}+\left[
2\left(  \frac{\lambda_{1}+\lambda_{3}}{\lambda_{1}-\lambda_{3}}\right)
-\frac{1}{3}\right]  h_{112}^{2}+\left(  \frac{\lambda_{1}+\lambda_{3}%
}{\lambda_{1}-\lambda_{3}}-\frac{1}{3}\right)  h_{113}^{2}\right\} \\
\geq0,
\end{gather*}
where we again used $\lambda_{1}+2\lambda_{3}>0$ from (\ref{l_1+2l_3>0}). We
have proved (\ref{jacobi-b2}) of Lemma 2.3.
\end{proof}

The following is the first main result of this section. This Jacobi inequality
is crucial in our proofs of Theorems 1.1 and 1.2.

\begin{proposition}
Let $u$ be a smooth solution to the special\ Lagrangian equation (\ref{EsLag})
with $n=3$ and $\Theta\geq\pi/2$ on $B_{R}.$ Set
\[
b=\max\left\{  \ln\sqrt{1+\lambda_{\max}^{2}},\ K\right\}  ,
\]
with $K=1+\ln\sqrt{1+\tan^{2}\left(  \frac{\Theta}{3}\right)  .}$ Then $b$
satisfies the integral Jacobi inequality
\begin{equation}
\int_{B_{R}}-\left\langle \nabla_{g}\varphi,\nabla_{g}b\right\rangle
_{g}dv_{g}\geq\frac{1}{3}\int_{B_{R}}\varphi\left\vert \nabla_{g}b\right\vert
^{2}dv_{g} \label{IJacobi-3}%
\end{equation}
for all non-negative $\varphi\in C_{0}^{\infty}\left(  B_{R}\right)  .$
\label{PIJacobi}
\end{proposition}

\begin{proof}
If $b_{1}=\ln\sqrt{1+\lambda_{\max}^{2}}$ is smooth everywhere, then the
pointwise Jacobi inequality (\ref{Jacobi-3d}) in Lemma 2.2 \ already implies
the integral Jacobi (\ref{IJacobi-3}).\ It is known that $\lambda_{\max}$ is
always a Lipschitz function of the entries of the Hessian $D^{2}u.$ \ Now $u$
is smooth in $x,$ so $b_{1}=\ln\sqrt{1+\lambda_{\max}^{2}}$ is Lipschitz in
terms of $x.$ If \ $b_{1}$ (or equivalently $\lambda_{\max})\;$is not smooth,
then the first two largest eigenvalues $\lambda_{1}\left(  x\right)  $ and
$\lambda_{2}\left(  x\right)  $ coincide, and $b_{1}\left(  x\right)
=b_{2}\left(  x\right)  ,$ where $b_{2}\left(  x\right)  $ is the average
$b_{2}=\left(  \ln\sqrt{1+\lambda_{1}^{2}}+\ln\sqrt{1+\lambda_{2}^{2}}\right)
/2.$ \ We prove the integral Jacobi inequality (\ref{IJacobi-3}) for a
possibly singular $b_{1}\left(  x\right)  $ in two cases. Set
\[
S=\left\{  x|\ \lambda_{1}\left(  x\right)  =\lambda_{2}\left(  x\right)
\right\}  .
\]

Case 1. $S$ has measure zero. For small $\tau>0,$ let
\begin{align*}
\Omega &  =B_{R}\backslash\left\{  x|\ b_{1}\left(  x\right)  \leq K\right\}
=B_{R}\backslash\left\{  x|\ b\left(  x\right)  =K\right\} \\
\Omega_{1}\left(  \tau\right)   &  =\left\{  x|\ b\left(  x\right)
=b_{1}\left(  x\right)  >b_{2}\left(  x\right)  +\tau\right\}  \cap\Omega\\
\Omega_{2}\left(  \tau\right)   &  =\left\{  x|\ b_{2}\left(  x\right)  \leq
b\left(  x\right)  =b_{1}\left(  x\right)  <b_{2}\left(  x\right)
+\tau\right\}  \cap\Omega.
\end{align*}
Now $b\left(  x\right)  =b_{1}\left(  x\right)  $ is smooth in $\overline
{\Omega_{1}\left(  \tau\right)  }.$ We claim that $b_{2}\left(  x\right)  $ is
smooth in $\overline{\Omega_{2}\left(  \tau\right)  }.$ We know $b_{2}\left(
x\right)  $ is smooth wherever $\lambda_{2}\left(  x\right)  >\lambda
_{3}\left(  x\right)  .$ If (the \ Lipschitz) $b_{2}\left(  x\right)  $ is not
smooth at $x_{\ast}\in\overline{\Omega_{2}\left(  \tau\right)  },$ then
\begin{align*}
\ln\sqrt{1+\lambda_{3}^{2}}  &  =\ln\sqrt{1+\lambda_{2}^{2}}\geq\ln
\sqrt{1+\lambda_{1}^{2}}-2\tau\\
&  \geq\ln\sqrt{1+\tan^{2}\left(  \frac{\Theta}{3}\right)  }+1-2\tau,
\end{align*}
by the choice of $K.$ For small enough $\tau$, we have $\lambda_{2}%
=\lambda_{3}>\tan\left(  \frac{\Theta}{3}\right)  $ and a contradiction
\[
\left(  \theta_{1}+\theta_{2}+\theta_{3}\right)  \left(  x_{\ast}\right)
>\Theta.
\]

Note that
\begin{gather*}
\int_{B_{R}}-\left\langle \nabla_{g}\varphi,\nabla_{g}b\right\rangle
_{g}dv_{g}=\int_{\Omega}-\left\langle \nabla_{g}\varphi,\nabla_{g}%
b\right\rangle _{g}dv_{g}\\
=\lim_{\tau\rightarrow0^{+}}\left[  \int_{\Omega_{1}\left(  \tau\right)
}-\left\langle \nabla_{g}\varphi,\nabla_{g}b\right\rangle _{g}dv_{g}%
+\int_{\Omega_{2}\left(  \tau\right)  }-\left\langle \nabla_{g}\varphi
,\nabla_{g}\left(  b_{2}+\tau\right)  \right\rangle _{g}dv_{g}\right]  .
\end{gather*}
By the smoothness of $b$ in $\Omega_{1}\left(  \tau\right)  $ and $b_{2}$ in
$\Omega_{2}\left(  \tau\right)  ,$ and also inequalities (\ref{Jacobi-3d}) and
(\ref{subharmonicb2}), we have
\begin{align*}
&  \int_{\Omega_{1}\left(  \tau\right)  }-\left\langle \nabla_{g}%
\varphi,\nabla_{g}b\right\rangle _{g}dv_{g}+\int_{\Omega_{2}\left(
\tau\right)  }-\left\langle \nabla_{g}\varphi,\nabla_{g}\left(  b_{2}%
+\tau\right)  \right\rangle _{g}dv_{g}\\
&  =\int_{\partial\Omega_{1}\left(  \tau\right)  }-\varphi\text{ }%
\partial_{\nu_{g}^{1}}b\text{ }dA_{g}+\int_{\Omega_{1}\left(  \tau\right)
}\varphi\bigtriangleup_{g}b_{1}dv_{g}\\
&  +\int_{\partial\Omega_{2}\left(  \tau\right)  }-\varphi\partial_{\nu
_{g}^{2}}\left(  b_{2}+\tau\right)  dA_{g}+\int_{\Omega_{2}\left(
\tau\right)  }\varphi\bigtriangleup_{g}\left(  b_{2}+\tau\right)  dv_{g}\\
&  \geq\int_{\partial\Omega_{1}\left(  \tau\right)  }-\varphi\text{ }%
\partial_{\nu_{g}^{1}}b\text{ }dA_{g}+\int_{\partial\Omega_{2}\left(
\tau\right)  }-\varphi\partial_{\nu_{g}^{2}}\left(  b_{2}+\tau\right)
dA_{g}+\frac{1}{3}\int_{\Omega_{1}\left(  \tau\right)  }\varphi\left\vert
\nabla_{g}b_{1}\right\vert ^{2}dv_{g},
\end{align*}
where $\nu_{g}^{1}$ and $\nu_{g}^{2}$ are the outward co-normals of
$\partial\Omega_{1}\left(  \tau\right)  $ and $\partial\Omega_{2}\left(
\tau\right)  $ with respect to the metric $g.$

Observe that if $b_{1}$ is not smooth on any part of $\partial\Omega
\backslash\partial B_{R}$, which is the $K$-level set of $b_{1,}$ then on this
portion $\partial\Omega\backslash\partial B_{R}$ is also the $K$-level set of
$b_{2},$ which is smooth near this portion. Applying Sard's theorem, we can
perturb $K$ so that $\partial\Omega$ is piecewise $C^{1}.$ Applying Sard's
theorem again, we find a subsequence of positive $\tau$ going to $0,$ so that
the boundaries $\partial\Omega_{1}\left(  \tau\right)  $ and $\partial
\Omega_{2}\left(  \tau\right)  $ are piecewise $C^{1}.$

Then, we show the above boundary integrals are non-negative. The boundary
integral portion along $\partial\Omega$ is easily seen non-negative, because
either $\varphi=0$, or $-\partial_{\nu_{g}^{1}}b\geq0,\ -\partial_{\nu_{g}%
^{2}}\left(  b_{2}+\tau\right)  $ $\geq0$ there. The boundary integral portion
in the interior of $\Omega$ is also non-negative, because there we have
\begin{gather*}
b=b_{2}+\tau\ \ \ \text{(and }b\geq b_{2}+\tau\ \ \text{in }\Omega_{1}\left(
\tau\right)  \text{)}\\
-\partial_{\nu_{g}^{1}}b\ -\partial_{\nu_{g}^{2}}\left(  b_{2}+\tau\right)
=\partial_{\nu_{g}^{2}}b\ -\partial_{\nu_{g}^{2}}\left(  b_{2}+\tau\right)
\geq0.
\end{gather*}
Taking the limit along the (Sard) sequence of $\tau$ going to $0,$ we obtain
$\Omega_{1}\left(  \tau\right)  \rightarrow\Omega$ up to a set of measure
zero, and
\begin{align*}
&  \int_{B_{R}}-\left\langle \nabla_{g}\varphi,\nabla_{g}b\right\rangle
_{g}dv_{g}\\
&  =\int_{\Omega}-\left\langle \nabla_{g}\varphi,\nabla_{g}b\right\rangle
_{g}dv_{g}\geq\frac{1}{3}\int_{\Omega}\left\vert \nabla_{g}b\right\vert
^{2}dv_{g}\\
&  =\frac{1}{3}\int_{B_{R}}\left\vert \nabla_{g}b\right\vert ^{2}dv_{g}.
\end{align*}

Case 2. $S$ has positive measure. The discriminant
\[
\mathcal{D}=\left(  \lambda_{1}-\lambda_{2}\right)  ^{2}\left(  \lambda
_{2}-\lambda_{3}\right)  ^{2}\left(  \lambda_{3}-\lambda_{1}\right)  ^{2}%
\]
is an analytic function in $B_{R},$ because the smooth $u$ is actually
analytic (cf. [M, p. 203]). So $\mathcal{D}$ must vanish identically. Then we
have either $\lambda_{1}\left(  x\right)  =\lambda_{2}\left(  x\right)  $ or
$\lambda_{2}\left(  x\right)  =\lambda_{3}\left(  x\right)  $ at any point
$x\in B_{R}.$ In turn, we know that $\lambda_{1}\left(  x\right)  =\lambda
_{2}\left(  x\right)  =\lambda_{3}\left(  x\right)  =\tan\frac{\Theta}{3}$ and
$b=K>b_{1}\left(  x\right)  $ at every \textquotedblleft
boundary\textquotedblright\ point of $S$ inside $B_{R},$ $x\in\partial
S\cap\mathring{B}_{R}.$ If the \textquotedblleft boundary\textquotedblright%
\ set $\partial S$ has positive measure, then $\lambda_{1}\left(  x\right)
=\lambda_{2}\left(  x\right)  =\lambda_{3}\left(  x\right)  =\tan\frac{\Theta
}{3}$ everywhere by the analyticity of $u,$ and (\ref{IJacobi-3}) is trivially
true. In the case that $\partial S$ has zero measure, $b=b_{1}>K$ is smooth up
to the boundary of every component of $\left\{  x|\ b\left(  x\right)
>K\right\}  .$ By the pointwise Jacobi inequality (\ref{jacobi-b2}), the
integral inequality (\ref{IJacobi-3}) is also valid in case 2.
\end{proof}

\subsection{Lewy rotation}

The next is the second main result of this section. Our proofs of Theorems 1.1
and 1.3 rely on this new representation of the original special\ Lagrangian graph.

\begin{proposition}
\label{Prop rotation}Let $u$ be a smooth solution to (\ref{EsLag}) with
$\Theta=$ $\left(  n-2\right)  \pi/2+\delta$ on $B_{R}(0)\subset\mathbb{R}%
^{n}.$\ Then the special Lagrangian surface $\mathfrak{M}=\left(
x,Du(x)\right)  $ can be represented as a gradient graph $\mathfrak{M}=\left(
\bar{x},D\bar{u}\left(  \bar{x}\right)  \right)  $ of the new potential
$\bar{u}$ satisfying (\ref{EsLag}) with phase $\Theta=\left(  n-2\right)
\pi/2$ in a domain containing a ball of radius
\[
\bar{R}\geq\frac{R}{2\cos\left(  \delta/n\right)  }.
\]

\end{proposition}

\begin{proof}
\ To obtain the new representation, we use a Lewy rotation (cf. [Y1], [Y2, p.
1356]). Take a $U\left(  n\right)  $ rotation of $\mathbb{C}^{n}%
\cong\mathbb{R}^{n}\times\mathbb{R}^{n}:\bar{z}=e^{-\sqrt{-1}\delta/n}z$ with
$z=x+\sqrt{-1}y$ and $\bar{z}=\bar{x}+\sqrt{-1}\bar{y}.$ Because $U\left(
n\right)  $ rotation preserves the length and complex structure,
$\mathfrak{M}$ is still a special Lagrangian submanifold with the
parametrization
\[
\left\{
\begin{array}
[c]{c}%
\bar{x}=x\cos\frac{\delta}{n}+Du\left(  x\right)  \sin\frac{\delta}{n}\\
D\bar{u}=-x\sin\frac{\delta}{n}+Du\left(  x\right)  \cos\frac{\delta}{n}%
\end{array}
\right.  .
\]
In order to show that this parameterization is that of a gradient graph over
$\bar{x}$ , we must first show that $\bar{x}(x)$ is a diffeomorphism onto its
image. This is accomplished by showing that%
\begin{equation}
\left\vert \bar{x}(x^{\alpha})-\bar{x}(x^{\beta})\right\vert \geq\frac
{1}{2\cos\delta/n}\left\vert x^{\alpha}-x^{\beta}\right\vert \label{lbonrad}%
\end{equation}
for any $x^{\alpha},$ $x^{\beta}.$ We assume by translation that $x^{\beta}=0$
and $Du\left(  x^{\beta}\right)  =0.$ Now\ $0<\delta<\pi,$ and $\theta
_{i}>\delta-\frac{\pi}{2},$ so $\ u+\frac{1}{2}\cot\delta x^{2}$ is convex,
\ and we have%
\begin{align*}
\left\vert \bar{x}\left(  x^{\alpha}\right)  -\bar{x}\left(  x^{\beta}\right)
\right\vert ^{2}  &  =\left\vert \bar{x}\left(  x^{\alpha}\right)  \right\vert
^{2}=\left\vert x^{\alpha}\cos\frac{\delta}{n}+Du\left(  x^{\alpha}\right)
\sin\frac{\delta}{n}\right\vert ^{2}\\
&  =\left\vert x^{\alpha}\left(  \cos\frac{\delta}{n}-\cot\delta\sin
\frac{\delta}{n}\right)  +\left[  Du\left(  x^{\alpha}\right)  +x^{\alpha}%
\cot\delta\right]  \sin\frac{\delta}{n}\right\vert ^{2}\\
&  =\left\vert x^{\alpha}\right\vert ^{2}\left[  \frac{\sin\frac{\left(
n-1\right)  \delta}{n}}{\sin\delta}\right]  ^{2}+\left\vert Du\left(
x^{\alpha}\right)  +x^{\alpha}\cot\delta\right\vert ^{2}\sin^{2}\frac{\delta
}{n}\\
&  +2\frac{\sin\frac{\left(  n-1\right)  \delta}{n}\sin\frac{\delta}{n}}%
{\sin\delta}\left\langle x^{\alpha},Du\left(  x^{\alpha}\right)  +x\cot
\delta\right\rangle \\
&  \geq\left\vert x^{\alpha}\right\vert ^{2}\left(  \frac{1}{2\cos\frac
{\delta}{n}}\right)  ^{2}.
\end{align*}
It follows that $\mathfrak{M}$ is a special Lagrangian graph over $\bar{x}$.
The Lagrangian graph is the gradient graph of a potential function $\bar{u}$
(cf. [HL, Lemma 2.2]), that is, $\mathfrak{M}=\left(  \bar{x},D\bar{u}\left(
\bar{x}\right)  \right)  .$ \ The eigenvalues $\bar{\lambda}_{i}$ of the
Hessian $D^{2}\bar{u}$ are determined by
\begin{equation}
\bar{\theta}_{i}=\arctan\bar{\lambda}_{i}=\theta_{i}-\frac{\delta}{n}%
\in\left(  -\frac{\pi}{2}+\frac{\left(  n-1\right)  \delta}{n}\,,\frac{\pi}%
{2}-\frac{\delta}{n}\right)  . \label{automatic bound}%
\end{equation}
Then
\[
\sum_{i=1}^{n}\bar{\theta}_{i}=\frac{\left(  n-2\right)  \pi}{2n},
\]
that is, $\bar{u}$ satisfies the special Lagrangian equation (\ref{EsLag}) of
phase $\left(  n-2\right)  \pi/2.$ \ The lower bound on $\bar{R}$ follows
immediately from (\ref{lbonrad}). \ 
\end{proof}

\subsection{Relative isoperimetric inequality}

We end with the last main result of this section, Proposition 2.3. This
relative isoperimetric inequality is needed in the proof of Theorem 1.2 to
prove a key ingredient, namely a Sobolev inequality for functions without
compact support. Proposition 2.3 is proved from the following classical
relative isoperimetric inequality for balls.

\begin{lemma}
\label{Lem poincare 2}Let $A$ and $A^{c}$ are disjoint measurable sets such
that $A\cup A^{c}=B_{1}(0)\subset\mathbb{R}^{n}.$\ Then
\begin{equation}
\min\left\{  \left\vert A\right\vert ,\left\vert A^{c}\right\vert \right\}
\leq C(n)\left\vert \partial A\cap\partial A^{c}\right\vert ^{n/n-1}.
\label{classical iso}%
\end{equation}

\end{lemma}

\begin{proof}
See for example [LY, Theorem 5.3.2.].
\end{proof}

\begin{proposition}
\label{Prop relative isoperimetric}\ Let $\Omega_{1}\subset$ $\Omega
_{2}\subset B_{\rho}\subset\mathbb{R}^{n}.$ \ Suppose that dist$(\Omega
_{1},\partial\Omega_{2})\geq2,$ also $A$ and $A^{c}$ are disjoint measurable
sets such that $A\cup A^{c}=\Omega_{2}.$ \ Then
\[
\min\left\{  \left\vert A\cap\Omega_{1}\right\vert ,\left\vert A^{c}\cap
\Omega_{1}\right\vert \right\}  \leq C(n)\rho^{n}\left\vert \partial
A\cap\partial A^{c}\right\vert ^{n/n-1}.
\]

\end{proposition}

\begin{proof}
\ Define a continuous function on $\Omega_{1}$
\[
\chi(x)=\frac{\left\vert A\cap B_{1}(x)\right\vert }{\left\vert B_{1}%
(x)\right\vert }.
\]
First, suppose that $\chi(x^{\ast})=1/2$ \ for some $x^{\ast}\in\Omega_{1}.$
From the relative isoperimetric inequality for balls (\ref{classical iso})
\[
\frac{\left\vert B_{1}(x^{\ast})\right\vert }{2}\leq C(n)\left\vert \partial
A\cap\partial A^{c}\cap B_{1}(x^{\ast})\right\vert ^{n/n-1}\leq C(n)\left\vert
\partial A\cap\partial A^{c}\right\vert ^{n/n-1}.
\]
Now $\ $%
\[
\min\left\{  \left\vert A\cap\Omega_{1}\right\vert ,\left\vert A^{c}\cap
\Omega_{1}\right\vert \right\}  \leq\left\vert \Omega_{1}\right\vert
<\left\vert B_{\rho}\right\vert =\frac{\left\vert B_{1}(x^{\ast})\right\vert
}{2}2\rho^{n}\leq C(n)\rho^{n}\left\vert \partial A\cap\partial A^{c}%
\right\vert ^{n/n-1},
\]
and the conclusion of this proposition follows.

On the other hand, suppose that for all $x\in\Omega_{1},$ $\chi(x)\neq1/2$.
\ Then either $\chi(x)<1/2$ on $\Omega_{1},$ or $\chi(x)>1/2$ on $\Omega_{1}.$
\ Without loss of generality we assume that $\chi(x)<1/2$ on $\Omega_{1}%
.$\ Cover $\Omega_{2}$ by $C(n)\rho^{n}$ balls of unit radius, $B_{1}(x_{i})$.
\ Consider the subcover which covers $\Omega_{1}$; each ball in this subcover
is completely contained inside $\Omega_{2}.$ \ Thus we may apply
(\ref{classical iso}) to each ball in this subcover and obtain
\[
\left\vert A\cap B_{1}(x_{i})\right\vert =\min\left\{  \left\vert A\cap
B_{1}(x_{i})\right\vert ,\left\vert A^{c}\cap B_{1}(x_{i})\right\vert
\right\}  \leq C(n)\left\vert \partial A\cap\partial A^{c}\right\vert
^{n/n-1}.
\]
Summing this inequality over the subcover, we get
\[
|A\cap\Omega_{1}|\leq\sum_{i=1}^{C(n)\rho^{n}}\left\vert A\cap B_{1}%
(x_{i})\right\vert \leq C(n)\rho^{n}C(n)\left\vert \partial A\cap\partial
A^{c}\right\vert ^{n/n-1}.
\]
Again, the conclusion of this proposition follows.
\end{proof}

\textbf{Remark}. Considering dumbbell type regions, we see that the order of
dependence on $\rho$ is sharp in Proposition 2.3.

\section{Proof Of Theorem 1.2}

For completeness, we reproduce the proof of Theorem 1.2 here. We assume that
$R=4$ and $u$ is a solution on $B_{4}\subset\mathbb{R}^{3}$ for simplicity of
notation. By scaling $u\left(  \frac{R}{4}x\right)  /\left(  \frac{R}%
{4}\right)  ^{2},$ we still get the estimate in Theorem 1.2. By symmetry, we
assume without loss of generality that $\Theta=\pi/2.$

Step 1. By the integral Jacobi inequality (\ref{IJacobi-3}) in Proposition
\ref{PIJacobi}, $b$\ is subharmonic in the integral sense, then $b^{3}$ is
also subharmonic in the integral sense on the minimal surface $\mathfrak{M}%
=\left(  x,Du\right)  :$%
\begin{align}
\int-\left\langle \nabla_{g}\varphi,\nabla_{g}b^{3}\right\rangle _{g}dv_{g}
&  =\int-\left\langle \nabla_{g}\left(  3b^{2}\varphi\right)  -6b\varphi
\nabla_{g}b,\nabla_{g}b\right\rangle _{g}dv_{g}\nonumber\\
&  \geq\int\left(  \varphi b^{2}\left\vert \nabla_{g}b\right\vert
^{2}+6b\varphi\left\vert \nabla_{g}b\right\vert ^{2}\right)  dv_{g}\geq0
\label{convex function also subharmonic}%
\end{align}
for all non-negative $\varphi\in C_{0}^{\infty},$ approximating $b$ by smooth
functions if necessary.

Applying Michael-Simon's mean value inequality [MS, Theorem 3.4] to the
Lipschitz subharmonic function $b^{3},$ we obtain
\[
b\left(  0\right)  \leq C\left(  3\right)  \left(  \int_{\mathfrak{B}_{1}%
\cap\mathfrak{M}}b^{3}dv_{g}\right)  ^{1/3}\leq C\left(  3\right)  \left(
\int_{B_{1}}b^{3}dv_{g}\right)  ^{1/3},
\]
where $\mathfrak{B}_{r}$ is the ball with radius $r$ and center $\left(
0,Du\left(  0\right)  \right)  $ in $\mathbb{R}^{3}\times\mathbb{R}^{3}$, and
$B_{r}$ is the ball with radius $r$ and center $0$ in $\mathbb{R}^{3}.$ Choose
a cut-off function $\varphi\in C_{0}^{\infty}\left(  B_{2}\right)  $ such that
$\varphi\geq0,$ $\varphi=1$ on $B_{1},$ and $\left\vert D\varphi\right\vert
\leq1.1,$ we then have
\[
\left(  \int_{B_{1}}b^{3}dv_{g}\right)  ^{1/3}\leq\left(  \int_{B_{2}}%
\varphi^{6}b^{3}dv_{g}\right)  ^{1/3}=\left(  \int_{B_{2}}\left(  \varphi
b^{1/2}\right)  ^{6}dv_{g}\right)  ^{1/3}.
\]
Applying the Sobolev inequality on the minimal surface $\mathfrak{M}$ [MS,
Theorem 2.1] or [A, Theorem 7.3] to $\varphi b^{1/2},$ which we may assume to
be $C^{1}$ by approximation, we obtain
\[
\left(  \int_{B_{2}}\left(  \varphi b^{1/2}\right)  ^{6}dv_{g}\right)
^{1/3}\leq C\left(  3\right)  \int_{B_{2}}\left\vert \nabla_{g}\left(  \varphi
b^{1/2}\right)  \right\vert ^{2}dv_{g}.
\]
Splitting the integrand as follows
\begin{align*}
\left\vert \nabla_{g}\left(  \varphi b^{1/2}\right)  \right\vert ^{2}  &
=\left\vert \frac{1}{2b^{1/2}}\varphi\nabla_{g}b+b^{1/2}\nabla_{g}%
\varphi\right\vert ^{2}\leq\frac{1}{2b}\varphi^{2}\left\vert \nabla
_{g}b\right\vert ^{2}+2b\left\vert \nabla_{g}\varphi\right\vert ^{2}\\
&  \leq\frac{1}{2}\varphi^{2}\left\vert \nabla_{g}b\right\vert ^{2}%
+2b\left\vert \nabla_{g}\varphi\right\vert ^{2},
\end{align*}
where we used $b\geq1,$ we get
\begin{align*}
b\left(  0\right)   &  \leq C\left(  3\right)  \int_{B_{2}}\left\vert
\nabla_{g}\left(  \varphi b^{1/2}\right)  \right\vert ^{2}dv_{g}\\
&  \leq C\left(  3\right)  \left(  \int_{B_{2}}\varphi^{2}\left\vert
\nabla_{g}b\right\vert ^{2}dv_{g}+\int_{B_{2}}b\left\vert \nabla_{g}%
\varphi\right\vert ^{2}dv_{g}\right) \\
&  \leq\underset{\text{Step\ 2}}{\underbrace{C\left(  3\right)  \left\Vert
Du\right\Vert _{L^{\infty}\left(  B_{2}\right)  }}}+C\left(  3\right)
\underset{\text{step\ 3}}{\underbrace{\left[  \left\Vert Du\right\Vert
_{L^{\infty}\left(  B_{3}\right)  }^{2}+\left\Vert Du\right\Vert _{L^{\infty
}\left(  B_{4}\right)  }^{3}\right]  }}.
\end{align*}

Step 2. By (\ref{IJacobi-3}) in Proposition \ref{PIJacobi}, $b$ satisfies the
Jacobi inequality in the integral sense:
\[
3\bigtriangleup_{g}b\geq\left|  \nabla_{g}b\right|  ^{2}.
\]
Multiplying both sides by the above non-negative cut-off function $\varphi\in
C_{0}^{\infty}\left(  B_{2}\right)  ,$ then integrating, we obtain
\begin{align*}
\int_{B_{2}}\varphi^{2}\left|  \nabla_{g}b\right|  ^{2}dv_{g}  &  \leq
3\int_{B_{2}}\varphi^{2}\bigtriangleup_{g}bdv_{g}\\
&  =-3\int_{B_{2}}\left\langle 2\varphi\nabla_{g}\varphi,\nabla_{g}%
b\right\rangle dv_{g}\\
&  \leq\frac{1}{2}\int_{B_{2}}\varphi^{2}\left|  \nabla_{g}b\right|
^{2}dv_{g}+18\int_{B_{2}}\left|  \nabla_{g}\varphi\right|  ^{2}dv_{g}.
\end{align*}
It follows that
\begin{equation}
\int_{B_{2}}\varphi^{2}\left|  \nabla_{g}b\right|  ^{2}dv_{g}\leq36\int
_{B_{2}}\left|  \nabla_{g}\varphi\right|  ^{2}dv_{g}. \label{jacobi outcome}%
\end{equation}
Observe the (``conformality'')\ identity:
\[
\left(  \frac{1}{1+\lambda_{1}^{2}},\frac{1}{1+\lambda_{2}^{2}},\frac
{1}{1+\lambda_{3}^{2}}\right)  V=\left(  \sigma_{1}-\lambda_{1,}\text{
\ }\sigma_{1}-\lambda_{2,}\text{ \ }\sigma_{1}-\lambda_{3}\right)
\]
where we used the identity $V=%
%TCIMACRO{\dprod \limits_{i=1}^{3}}%
%BeginExpansion
{\displaystyle\prod\limits_{i=1}^{3}}
%EndExpansion
\sqrt{\left(  1+\lambda_{i}^{2}\right)  }=\sigma_{1}-\sigma_{3}$ with
$\sigma_{2}=1$. \ We then have
\begin{align}
\left|  \nabla_{g}\varphi\right|  ^{2}dv_{g}  &  =\sum_{i=1}^{3}\frac{\left(
D_{i}\varphi\right)  ^{2}}{1+\lambda_{i}^{2}}Vdx=\sum_{i=1}^{3}\left(
D_{i}\varphi\right)  ^{2}\left(  \sigma_{1}-\lambda_{i}\right)
dx\label{conformality}\\
&  \leq2.42\bigtriangleup u\ dx.\nonumber
\end{align}
Thus
\begin{gather*}
\int_{B_{2}}\varphi^{2}\left|  \nabla_{g}b\right|  ^{2}dv_{g}\leq C\left(
3\right)  \int_{B_{2}}\bigtriangleup u\ dx\\
\leq C\left(  3\right)  \left\|  Du\right\|  _{L^{\infty}\left(  B_{2}\right)
}.
\end{gather*}

Step 3. By (\ref{conformality}), we get
\[
\int_{B_{2}}b\left\vert \nabla_{g}\varphi\right\vert ^{2}dv_{g}\leq C\left(
3\right)  \int_{B_{2}}b\bigtriangleup u\ dx.
\]
Choose another cut-off function $\psi\in C_{0}^{\infty}\left(  B_{3}\right)  $
such that $\psi\geq0,$ $\psi=1$ on $B_{2},$ and $\left\vert D\psi\right\vert
\leq1.1.$ We have
\begin{align*}
\int_{B_{2}}b\bigtriangleup udx  &  \leq\int_{B_{3}}\psi b\bigtriangleup
udx=\int_{B_{3}}-\left\langle bD\psi+\psi Db,Du\right\rangle dx\\
&  \leq\left\Vert Du\right\Vert _{L^{\infty}\left(  B_{3}\right)  }\int
_{B_{3}}\left(  b\left\vert D\psi\right\vert +\psi\left\vert Db\right\vert
\right)  dx\\
&  \leq C\left(  3\right)  \left\Vert Du\right\Vert _{L^{\infty}\left(
B_{3}\right)  }\int_{B_{3}}\left(  b+\left\vert Db\right\vert \right)  dx.
\end{align*}
Now
\[
b=\max\left\{  \ln\sqrt{1+\lambda_{\max}^{2}},\ K\right\}  \leq\lambda_{\max
}+K<\lambda_{1}+\lambda_{2}+\lambda_{3}+K=\bigtriangleup u+K,
\]
where $\lambda_{2}+\lambda_{3}>0$ follows from $\arctan\lambda_{2}%
+\arctan\lambda_{3}=\frac{\pi}{2}-\arctan\lambda_{1}>0.$ Hence
\[
\int_{B_{3}}bdx\leq C(3)(1+\left\Vert Du\right\Vert _{L^{\infty}\left(
B_{3}\right)  }).
\]
And we have left to estimate $\int_{B_{3}}\left\vert Db\right\vert dx:$
\begin{align*}
\ \int_{B_{3}}\left\vert Db\right\vert dx  &  \leq\int_{B_{3}}\sqrt{\sum
_{i=1}^{3}\frac{\left(  b_{i}\right)  ^{2}}{\left(  1+\lambda_{i}^{2}\right)
}\left(  1+\lambda_{1}^{2}\right)  \left(  1+\lambda_{2}^{2}\right)  \left(
1+\lambda_{3}^{2}\right)  }\ dx\\
&  =\int_{B_{3}}\left\vert \nabla_{g}b\right\vert Vdx\\
&  \leq\left(  \int_{B_{3}}\left\vert \nabla_{g}b\right\vert ^{2}Vdx\right)
^{1/2}\left(  \int_{B_{3}}Vdx\right)  ^{1/2}.
\end{align*}
Repeating the \textquotedblleft Jacobi\textquotedblright\ argument from Step
2, we see
\[
\int_{B_{3}}\left\vert \nabla_{g}b\right\vert ^{2}Vdx\leq C\left(  3\right)
\left\Vert Du\right\Vert _{L^{\infty}\left(  B_{4}\right)  }.
\]
Then by the Sobolev inequality on the minimal surface $\mathfrak{M},$ we have
\[
\int_{B_{3}}Vdx=\int_{B_{3}}dv_{g}\leq\int_{B_{4}}\phi^{6}dv_{g}\leq C\left(
3\right)  \left(  \int_{B_{4}}\left\vert \nabla_{g}\phi\right\vert ^{2}%
dv_{g}\right)  ^{3},
\]
where the non-negative cut-off function $\phi\in C_{0}^{\infty}\left(
B_{4}\right)  $ satisfies $\phi=1$ on $B_{3},$ and $\left\vert D\phi
\right\vert \leq1.1.$ Applying the conformality equality (\ref{conformality})
again, we obtain
\[
\int_{B_{4}}\left\vert \nabla_{g}\phi\right\vert ^{2}dv_{g}\leq C\left(
3\right)  \int_{B_{4}}\bigtriangleup u\ dx\leq C\left(  3\right)  \left\Vert
Du\right\Vert _{L^{\infty}\left(  B_{4}\right)  }.
\]
Thus we get
\[
\int_{B_{3}}Vdx\leq C\left(  3\right)  \left\Vert Du\right\Vert _{L^{\infty
}\left(  B_{4}\right)  }^{3}%
\]
and
\[
\int_{B_{3}}\left\vert Db\right\vert dx\leq C\left(  3\right)  \left\Vert
Du\right\Vert _{L^{\infty}\left(  B_{4}\right)  }^{2}.
\]
In turn, we obtain
\[
\int_{B_{2}}b\left\vert \nabla_{g}\varphi\right\vert ^{2}dv_{g}\leq C\left(
3\right)  \left[  K\left\Vert Du\right\Vert _{L^{\infty}\left(  B_{3}\right)
}+\left\Vert Du\right\Vert _{L^{\infty}\left(  B_{3}\right)  }^{2}+\left\Vert
Du\right\Vert _{L^{\infty}\left(  B_{4}\right)  }^{3}\right]  .
\]
Finally collecting all the estimates in the above three steps, we arrive at
\begin{align}
\lambda_{\max}\left(  0\right)   &  \leq\exp\left[  C\left(  3\right)  \left(
\left\Vert Du\right\Vert _{L^{\infty}\left(  B_{4}\right)  }+\left\Vert
Du\right\Vert _{L^{\infty}\left(  B_{4}\right)  }^{2}+\left\Vert Du\right\Vert
_{L^{\infty}\left(  B_{4}\right)  }^{3}\right)  \right]
\label{sigma2 estimate}\\
&  \leq C\left(  3\right)  \exp\left[  C\left(  3\right)  \left\Vert
Du\right\Vert _{L^{\infty}\left(  B_{4}\right)  }^{3}\right]  .\nonumber
\end{align}
This completes the proof of Theorem 1.2.

\section{Proof Of Theorem 1.1}

As in the proof of Theorem 1.2, we assume that $R=8$ and $u$ is a solution on
$B_{8}\subset\mathbb{R}^{3}$ for simplicity of notation. By scaling $v\left(
x\right)  =u\left(  \frac{R}{8}x\right)  /\left(  \frac{R}{8}\right)  ^{2},$
we still get the estimate in Theorem 1.1. We consider the cases when
$\Theta=\pi/2+\delta$ for $\delta\in\left(  0,\pi\right)  .$\ The cases
$\Theta<-\pi/2$ follow by symmetry. \ 

Step 1. As preparation for the proof of Theorem 1.2, we take the phase $\pi/2$
representation $\mathfrak{M}=\left(  \bar{x},D\bar{u}\left(  \bar{x}\right)
\right)  $ in Proposition \ref{Prop rotation} for the original special
Lagrangian graph $\mathfrak{M}=\left(  x,Du\left(  x\right)  \right)  $ with
$x\in B_{8}.$ The \textquotedblleft critical\textquotedblright\ representation
is\
\begin{equation}
\left\{
\begin{array}
[c]{c}%
\bar{x}=x\cos\frac{\delta}{3}+Du\left(  x\right)  \sin\frac{\delta}{3}\\
D\bar{u}=-x\sin\frac{\delta}{3}+Du\left(  x\right)  \cos\frac{\delta}{3}%
\end{array}
\right.  . \label{rotation}%
\end{equation}
Define
\[
\bar{\Omega}_{r}=\bar{x}(B_{r}(0)).
\]
Then we have from (\ref{lbonrad})
\begin{equation}
\text{dist}(\bar{\Omega}_{1},\partial\bar{\Omega}_{5})\geq\frac{4}{2\cos
\delta/3}>2. \label{distance}%
\end{equation}
We see from (\ref{rotation}) that $\left\vert \bar{x}\right\vert \leq\rho$ for
$\bar{x}\in$ $\bar{\Omega}_{8}$ with
\begin{equation}
\rho=8\cos\frac{\delta}{3}+\left\Vert Du\right\Vert _{L^{\infty}\left(
B_{8}\right)  }\sin\frac{\delta}{3} \label{define M}%
\end{equation}
and that $\left\vert D\bar{u}\left(  \bar{x}\right)  \right\vert \leq\kappa$
(for $\bar{x}\in\bar{\Omega}_{8}$) with
\begin{equation}
\kappa=8\sin\frac{\delta}{3}+\left\Vert Du\right\Vert _{L^{\infty}\left(
B_{8}\right)  }\cos\frac{\delta}{3}. \label{define H}%
\end{equation}
The eigenvalues of the new potential $\bar{u}$ satisfy (\ref{automatic bound})
and the interior Hessian bound by Theorem 1.2%
\[
\left\vert D^{2}\bar{u}\left(  \bar{x}\right)  \right\vert \leq C\left(
3\right)  \exp\left[  C\left(  3\right)  \kappa^{3}\right]
\]
for $\bar{x}\in\bar{\Omega}_{5}.$ It follows that the induced metric on
$\mathfrak{M}=(\bar{x},D\bar{u}(\bar{x}))\,$in $\bar{x}$ coordinates is
bounded on $\bar{\Omega}_{5}$ by%
\begin{equation}
d\bar{x}^{2}\leq\bar{g}\left(  \bar{x}\right)  \leq\mu(\kappa,\delta)d\bar
{x}^{2}, \label{metric bound h}%
\end{equation}
where%
\begin{equation}
\mu(\kappa,\delta)=\min\left\{  1+C(3)\exp\left[  C(3)\kappa^{3}\right]
,\left[  1+\left(  \cot\frac{\delta}{3}\right)  ^{2}\right]  \right\}  .
\label{define h}%
\end{equation}

Step 2. Relying on the above set-up and the relative isoperimetric inequality
in Proposition \ref{Prop relative isoperimetric}, we proceed with the
following Sobolev inequality for functions without compact support.

\begin{proposition}
Let $u$ be a smooth solution to (\ref{EsLag}) with \ $\Theta=\pi/2+\delta$ on
$B_{5}(0)\subset$ $\mathbb{R}^{3}.$ Let $f$ be a smooth positive function on
the special Lagrangian surface $\mathfrak{M=}\left(  x,Du\left(  x\right)
\right)  .$ Then%
\[
\left[  \int_{B_{1}}\left\vert (f-\iota)^{+}\right\vert ^{3/2}dv_{g}\right]
^{2/3}\leq C(3)\rho^{4}\mu(\kappa,\delta)\int_{B_{5}}\left\vert \nabla
_{g}(f-\iota)^{+}\right\vert dv_{g},
\]
where $\rho,$ $\kappa,$ and $\mu$ were defined in (\ref{define M}),
(\ref{define H}), and (\ref{define h}); also $\iota=\int_{B_{5}(0)}fdx.$
\end{proposition}

\begin{proof}
Step 2.1. Let $M=||f||_{L^{\infty}(B_{1})}.$ We may assume $\iota<M.$ \ By
Sard's theorem, the level set $\{x|\ f(x)=t\}$ is $C^{1}$ for almost all $t.$
\ We first show that for all such $t$ $\in\lbrack\iota,M],$%
\begin{equation}
|\{x|\ f(x)>t\}\cap B_{1}|_{g}\leq C(3)\rho^{6}\left[  \mu(\kappa
,\delta)|\{x|\ f(x)=t\}\cap B_{5}|_{g}\right]  ^{3/2}. \label{isop}%
\end{equation}
(Here $\left\vert \ \right\vert _{g}$ and $\left\vert \ \right\vert _{\bar{g}%
}$ denote the area or volume with respect to the induced metric; $\left\vert
\ \right\vert $ denotes the ones with respect to the Euclidean metric as in
Lemma \ref{Lem poincare 2} and Proposition \ref{Prop relative isoperimetric} .)

From $t>\int_{B_{5}}fdx,$ it follows that $\left\vert \{x|\ f(x)>t\}\cap
B_{1}\right\vert <1$ and consequently
\begin{equation}
\left\vert \{x|\ f(x)\leq t\}\cap B_{1}\right\vert >\left\vert B_{1}%
\right\vert -1>1. \label{level set not small}%
\end{equation}
Now we use instead the coordinates for $\mathfrak{M}=(\bar{x},D\bar{u}(\bar
{x}))$ given by the Lewy rotation (\ref{rotation}). Let
\[
A_{t}=\{\bar{x}|\ f(\bar{x})>t\}\cap\bar{\Omega}_{5},
\]
where we are treating $f$ as a function on the special Lagrangian surface
$\mathfrak{M}.$ \ Applying Proposition \ref{Prop relative isoperimetric} with
(\ref{distance}) and (\ref{define M}), we see that
\[
\min\left\{  \left\vert A_{t}\cap\bar{\Omega}_{1}\right\vert ,\left\vert
A_{t}^{c}\cap\bar{\Omega}_{1}\right\vert \right\}  \leq C(3)\rho^{3}\left\vert
\partial A_{t}\cap\partial A_{t}^{c}\right\vert ^{3/2}.
\]
\ Now either $\left\vert A_{t}\cap\bar{\Omega}_{1}\right\vert \leq\left\vert
A_{t}^{c}\cap\bar{\Omega}_{1}\right\vert ,$ or vice versa.

If $\left\vert A_{t}\cap\bar{\Omega}_{1}\right\vert \leq\left\vert A_{t}%
^{c}\cap\bar{\Omega}_{1}\right\vert ,$ then we have from (\ref{metric bound h}%
)
\begin{align*}
\left\vert A_{t}\cap\bar{\Omega}_{1}\right\vert _{\bar{g}}  &  \leq\left[
\mu(\kappa,\delta)\right]  ^{3/2}\left\vert A_{t}\cap\bar{\Omega}%
_{1}\right\vert \\
&  \leq C(3)\rho^{3}\left[  \mu(\kappa,\delta)\right]  ^{3/2}\left\vert
\partial A_{t}\cap\partial A_{t}^{c}\right\vert ^{3/2}\\
&  \leq C(3)\rho^{3}\left[  \mu(\kappa,\delta)\right]  ^{3/2}\left\vert
\partial A_{t}\cap\partial A_{t}^{c}\right\vert _{\bar{g}}^{3/2}.
\end{align*}

Otherwise, if $\left\vert A_{t}\cap\bar{\Omega}_{1}\right\vert >\left\vert
A_{t}^{c}\cap\bar{\Omega}_{1}\right\vert ,$ still we have that
\[
\left\vert A_{t}\cap\bar{\Omega}_{1}\right\vert \leq C(3)\rho^{3}\left\vert
A_{t}^{c}\cap\bar{\Omega}_{1}\right\vert
\]
as $\left\vert A_{t}^{c}\cap\bar{\Omega}_{1}\right\vert \geq1/2^{3}$ from
(\ref{level set not small}) and (\ref{lbonrad}), and $\rho\geq8\cos\frac{\pi
}{3}$ from (\ref{define M}). Thus
\begin{align*}
\left\vert A_{t}\cap\bar{\Omega}_{1}\right\vert _{\bar{g}}  &  \leq\left[
\mu(\kappa,\delta)\right]  ^{3/2}C(3)\rho^{3}\left\vert A_{t}^{c}\cap
\bar{\Omega}_{1}\right\vert \\
&  \leq C(3)\rho^{6}\left[  \mu(\kappa,\delta)\right]  ^{3/2}\left\vert
\partial A_{t}\cap\partial A_{t}^{c}\right\vert _{\bar{g}}^{3/2}.
\end{align*}
In either case we have the desired isoperimetric inequality (now given in the
new coordinates for $\mathfrak{M}$ ) which holds for $\iota<t<M$
\[
\left\vert A_{t}\cap\bar{\Omega}_{1}\right\vert _{\bar{g}}\leq C(3)\rho
^{6}\left[  \mu(\kappa,\delta)\left\vert \partial A_{t}\cap\partial A_{t}%
^{c}\right\vert _{\bar{g}}\right]  ^{3/2},
\]
or equivalently (\ref{isop}) in the original coordinates.

Step 2.2. \ With this isoperimetric inequality in hand, the following proof is
standard (cf. [LY, Theorem 5.3.1]).
\begin{align*}
\left[  \int_{B_{1}}\left\vert (f-\iota)^{+}\right\vert ^{3/2}dv_{g}\right]
^{2/3}  &  =\left(  \int_{0}^{M-\iota}\left\vert \{x|\ f(x)-\iota>t\}\cap
B_{1}\right\vert _{g}dt^{3/2}\right)  ^{2/3}\\
&  \leq\int_{0}^{M-\iota}\left\vert \left\{  x|\ f(x)-\iota>t\}\cap
B_{1}\right\}  \right\vert _{g}^{2/3}dt\\
&  \leq C(3)\rho^{4}\mu(\kappa,\delta)\int_{0}^{M-\iota}|\{x|\ f(x)=t\}\cap
B_{5}|_{g}dt\\
&  \leq C(3)\rho^{4}\mu(\kappa,\delta)\int_{B_{5}}\left\vert \nabla
_{g}(f-\iota)^{+}\right\vert dv_{g},
\end{align*}
where the last inequality followed from the the coarea formula; the second
inequality from (\ref{isop}); and the first inequality from the
Hardy-Littlewood-Polya inequality for any nonnegative, nonincreasing integrand
$\eta\left(  t\right)  :$%
\[
\left[  \int_{0}^{T}\eta\left(  t\right)  ^{q}dt^{q}\right]  ^{1/q}\leq
\int_{0}^{T}\eta\left(  t\right)  dt.
\]
This H-L-P inequality (with $q>1$) is proved by noting that $s\eta\left(
s\right)  \leq\int_{0}^{s}\eta\left(  t\right)  dt$ and integrating the
inequality%
\[
q\left[  s\eta\left(  s\right)  \right]  ^{q-1}\eta\left(  s\right)  \leq
q\left[  \int_{0}^{s}\eta\left(  t\right)  dt\right]  ^{q-1}\eta\left(
s\right)  =\frac{d}{ds}\left[  \int_{0}^{s}\eta\left(  t\right)  dt\right]
^{q}.
\]
The proposition is thus proved.
\end{proof}

Step 3. We continue the proof of Theorem 1.1. As in the proof of Theorem 1.2,
we take
\[
b=\max\left\{  \ln\sqrt{1+\lambda_{\max}^{2}},\ K_{\Theta}\right\}
\ \ \ \text{with\ \ }\ K_{\Theta}=1+\ln\sqrt{1+\tan^{2}\left(  \frac{\Theta
}{3}\right)  }.
\]
Based on Proposition 2.1, a calculation similar to
(\ref{convex function also subharmonic}) shows that the Lipschitz function
$\left[  (b-\iota)^{+}\right]  ^{3/2}$ is weakly subharmonic, where
$\iota=\int_{B_{5}(0)}bdx.$ We apply Michael and Simon's mean value inequality
[MS, Theorem 3.4] to obtain
\begin{align*}
(b-\iota)^{+}(0)  &  \leq C(3)\left[  \int_{B_{1}}\left\vert (b-\iota
)^{+}\right\vert ^{3/2}dv_{g}\right]  ^{2/3}\\
&  \leq C(3)\rho^{4}\mu(\kappa,\delta)\int_{B_{5}}\left\vert \nabla
_{g}(b-\iota)^{+}\right\vert dv_{g},
\end{align*}
where the second inequality follows from Proposition 4.1, approximating
$(b-\iota)^{+}$ by smooth functions if necessary. Thus
\begin{align}
b(0)  &  \leq C(3)\rho^{4}\mu(\kappa,\delta)\int_{B_{5}}\left\vert \nabla
_{g}b\right\vert dv_{g}+\int_{B_{5}}bdx\nonumber\\
&  \leq C(3)\rho^{4}\mu(\kappa,\delta)\left(  \int_{B_{5}}\left\vert
\nabla_{g}b\right\vert ^{2}dv_{g}\right)  ^{1/2}\left(  \int_{B_{5}%
}Vdx\right)  ^{1/2}+\int_{B_{5}}Vdx\nonumber\\
&  \leq C(3)\rho^{4}\mu(\kappa,\delta)\int_{B_{6}}Vdx \label{b in terms of V}%
\end{align}
where we have used the Jacobi inequality in Proposition 2.1, and a similar
calculation leading to (\ref{jacobi outcome}) in the proof of \ Theorem 1.2. \ 

Step 4. We finish the proof of Theorem 1.1 by bounding $\int_{B_{6}}Vdx.$
Observe%
\[
V=\left\vert \left(  1+\sqrt{-1}\lambda_{1}\right)  \cdots\left(  1+\sqrt
{-1}\lambda_{3}\right)  \right\vert =\frac{\sigma_{2}-1}{\left\vert \cos
\Theta\right\vert }>0.
\]
We control the integral of $\sigma_{2}$ in the following.%
\begin{align*}
\int_{B_{r}}\sigma_{2}dx  &  =\int_{B_{r}}\frac{1}{2}\left[  \left(
\bigtriangleup u\right)  ^{2}-\left\vert D^{2}u\right\vert ^{2}\right]  dx\\
&  =\frac{1}{2}\int_{B_{r}}div\left[  \left(  \bigtriangleup uI-D^{2}u\right)
Du\right]  dx\\
&  =\frac{1}{2}\int_{\partial B_{r}}\left\langle \left(  \bigtriangleup
uI-D^{2}u\right)  Du,\nu\right\rangle dA,
\end{align*}
where $\nu$ is the outward normal of $B_{r}.$ Diagonalizing $D^{2}u,$ we see
easily that%
\[
\bigtriangleup uI-D^{2}u=\left[
\begin{array}
[c]{ccc}%
\lambda_{2}+\lambda_{3} &  & \\
& \lambda_{3}+\lambda_{1} & \\
&  & \lambda_{1}+\lambda_{2}%
\end{array}
\right]  >0
\]
as $\theta_{i}+\theta_{j}>0$ with $\theta_{1}+\theta_{2}+\theta_{3}\geq\pi/2.$
Then%
\[
\int_{B_{r}}\sigma_{2}dx\leq\left\Vert Du\right\Vert _{L^{\infty}\left(
\partial B_{r}\right)  }\int_{\partial B_{r}}\bigtriangleup udA.
\]
Integrating the boundary integral from $r=6$ to $r=7,$ we get%
\begin{align*}
\int_{B_{6}}\sigma_{2}dx  &  \leq\left\Vert Du\right\Vert _{L^{\infty}\left(
B_{7}\right)  }\min_{r\in\left[  6,7\right]  }\int_{\partial B_{r}%
}\bigtriangleup udA\\
&  \leq\left\Vert Du\right\Vert _{L^{\infty}\left(  B_{7}\right)  }\int
_{B_{7}}\bigtriangleup udx\\
&  \leq C\left(  3\right)  \left\Vert Du\right\Vert _{L^{\infty}\left(
B_{7}\right)  }^{2}.
\end{align*}
It follows that for $\Theta\geq\pi/2$%
\begin{align*}
\int_{B_{6}}Vdx  &  =\frac{1}{\left\vert \cos\Theta\right\vert }\int_{B_{6}%
}\left(  \sigma_{2}-1\right)  dx\leq\frac{1}{\left\vert \cos\Theta\right\vert
}\int_{B_{6}}\sigma_{2}dx\\
&  \leq\frac{C\left(  3\right)  }{\left\vert \cos\Theta\right\vert }\left\Vert
Du\right\Vert _{L^{\infty}\left(  B_{7}\right)  }^{2}%
\end{align*}
or%
\begin{equation}
\int_{B_{6}}Vdx\leq\frac{C\left(  3\right)  }{\left\vert \cos\Theta\right\vert
\left\Vert Du\right\Vert _{L^{\infty}\left(  B_{8}\right)  }}\left\Vert
Du\right\Vert _{L^{\infty}\left(  B_{8}\right)  }^{3}. \label{volume theta}%
\end{equation}

In order to get $\Theta$-independent control on the volume, we estimate the
volume in another way. By the Sobolev inequality on the minimal surface
$\mathfrak{M}$ [MS, Theorem 2.1] or [A, Theorem 7.3], we have%
\[
\int_{B_{6}}Vdx=\int_{B_{6}}dv_{g}\leq\int_{B_{7}}\phi^{6}dv_{g}\leq C\left(
3\right)  \left[  \int_{B_{7}}\left\vert \nabla_{g}\phi\right\vert ^{2}%
dv_{g}\right]  ^{3},
\]
where the nonnegative cut-off function $\phi\in C_{0}^{\infty}\left(
B_{7}\right)  $ satisfies $\phi=1$ on $B_{6}$ and $\left\vert D\phi\right\vert
\leq1.1.$

Observe the (conformality) identity again%
\begin{gather*}
\left(  \frac{1}{1+\lambda_{1}^{2}},\cdots,\frac{1}{1+\lambda_{3}^{2}}\right)
V=\\
\left(  \sin\Theta\left(  \sigma_{1}-\lambda_{1}\right)  +\cos\Theta\left(
1-\frac{\sigma_{3}}{\lambda_{1}}\right)  ,\cdots,\sin\Theta\left(  \sigma
_{1}-\lambda_{3}\right)  +\cos\Theta\left(  1-\frac{\sigma_{3}}{\lambda_{3}%
}\right)  \right)  ,
\end{gather*}
which follows from differentiating the complex identity%
\[
\ln V+\sqrt{-1}\sum_{i=1}^{3}\arctan\lambda_{i}=\ln\left[  1-\sigma_{2}%
+\sqrt{-1}\left(  \sigma_{1}-\sigma_{3}\right)  \right]  .
\]
We then have%
\begin{align*}
\int_{B_{7}}\left\vert \nabla_{g}\phi\right\vert ^{2}dv_{g} &  =\int_{B_{7}%
}\sum_{i=1}^{3}\frac{\left\vert \phi_{i}\right\vert ^{2}}{1+\lambda_{i}^{2}%
}Vdx\\
&  \leq1.21\int_{B_{7}}\left[  2\sin\Theta\sigma_{1}+\cos\Theta\left(
3-\sigma_{2}\right)  \right]  dx\\
&  \leq C\left(  3\right)  \left[  \left\vert \sin\Theta\right\vert \left\Vert
Du\right\Vert _{L^{\infty}\left(  B_{8}\right)  }+\left\vert \cos
\Theta\right\vert \left\Vert Du\right\Vert _{L^{\infty}\left(  B_{8}\right)
}^{2}\right]
\end{align*}
for $\Theta\geq\pi/2,$ where we used the argument leading to
(\ref{volume theta}). Thus we get%
\begin{equation}
\int_{B_{6}}Vdx\leq C\left(  3\right)  \left[  \left\vert \sin\Theta
\right\vert \left\Vert Du\right\Vert _{L^{\infty}\left(  B_{8}\right)
}+\left\vert \cos\Theta\right\vert \left\Vert Du\right\Vert _{L^{\infty
}\left(  B_{8}\right)  }\left\Vert Du\right\Vert _{L^{\infty}\left(
B_{8}\right)  }\right]  ^{3}\label{volume uniform}%
\end{equation}
Now either  $\left\vert \cos\Theta\right\vert \left\Vert Du\right\Vert
_{L^{\infty}\left(  B_{8}\right)  }\leq1$ or $\left\vert \cos\Theta\right\vert
\left\Vert Du\right\Vert _{L^{\infty}\left(  B_{8}\right)  }>1;$ combining
(\ref{volume theta}) and (\ref{volume uniform}), we have in either case%
\[
\int_{B_{6}}Vdx\leq C\left(  3\right)  \left\Vert Du\right\Vert _{L^{\infty
}\left(  B_{8}\right)  }^{3}.
\]

Finally from the above inequality and (\ref{b in terms of V}), we conclude
\begin{align*}
b(0)  &  \leq C(3)\rho^{4}\mu(\kappa,\delta)\left\Vert Du\right\Vert
_{L^{\infty}\left(  B_{8}\right)  }^{3}\\
&  \leq C(3)\rho^{4}\left\Vert Du\right\Vert _{L^{\infty}\left(  B_{8}\right)
}^{3}\min\left\{  1+C(3)\exp\left[  C(3)\kappa^{3}\right]  ,1+\left(
\cot\frac{\delta}{3}\right)  ^{2}\right\}  .
\end{align*}
Exponentiating, and recalling (\ref{define M}), (\ref{define H}) and
(\ref{define h}), we have the $\Theta$-independent bound
\[
\left\vert D^{2}u(0)\right\vert \leq C(3)\exp\left\{  C(3)\exp\left[
C(3)\left\Vert Du\right\Vert _{L^{\infty}\left(  B_{8}\right)  }^{3}\right]
\right\}
\]
and the $\Theta$-dependent bound
\[
\left\vert D^{2}u(0)\right\vert \leq C(3)\exp\left\{  C(3)\left[  1+\left(
\cot\frac{\delta}{3}\right)  ^{2}\right]  \left[  1+\left\Vert Du\right\Vert
_{L^{\infty}\left(  B_{8}\right)  }\sin\frac{\delta}{3}\right]  ^{4}\left\Vert
Du\right\Vert _{L^{\infty}\left(  B_{8}\right)  }^{3}\right\}  .
\]
Simplifying the above expressions, we arrive at the conclusion of Theorem 1.1.

\section{Proof of Theorem 1.3}

We assume that $R=1$ by scaling $u\left(  Rx\right)  /R^{2},$ and $\Theta
\geq\left(  n-2\right)  \pi/2$ by symmetry.

\textbf{Case }$\Theta=\left(  n-2\right)  \pi/2.$ Set $M=\hbox{{osc}}_{B_{1}%
}u.$ We may assume $M>0.$ By replacing $u$ with $u-\min_{B_{1}}u+M,$ we have
$M\leq u\leq2M$ in $B_{1}.$ Let
\[
w=\eta\left\vert Du\right\vert +Au^{2}%
\]
with $\eta=1-\left\vert x\right\vert ^{2}$ and $A=n/M.$ We assume that $w$
attains its maximum at an interior point $x^{\ast}\in B_{1},$ otherwise $w$
would take its maximum on the boundary $\partial B_{1}$ and the conclusion
would be straightforward. Choose a coordinate system so that $D^{2}u$ is
diagonalized at $x^{\ast}.$ We assume, say $u_{n}\geq\frac{\left\vert
Du\right\vert }{\sqrt{n}}\left(  >0\right)  $ at $x^{\ast}.$ For all
$i=1,\cdots,n,$ we have at $x^{\ast}$
\[
0=w_{i}=\eta\left\vert Du\right\vert _{i}+\eta_{i}\left\vert Du\right\vert
+2Auu_{i},
\]
then
\begin{equation}
\frac{u_{i}u_{ii}}{\left\vert Du\right\vert }=\left\vert Du\right\vert
_{i}=-\frac{\eta_{i}\left\vert Du\right\vert +2Auu_{i}}{\eta}.
\label{E critical gradient}%
\end{equation}
In particular, we have $u_{nn}<0$ by the choice of $A.$ \ Since the phase
$\Theta\geq\left(  n-2\right)  \pi/2,$ it follows that $\lambda_{n}%
=\lambda_{\min},$ $\left\vert \lambda_{n}\right\vert \leq\lambda_{k},$ and \
\begin{equation}
g^{nn}=\frac{1}{1+\lambda_{n}^{2}}\geq\frac{1}{1+\lambda_{k}^{2}}=g^{kk}
\label{Inequ g^nn}%
\end{equation}
for $k=1,\cdots,n-1$ at $x^{\ast}.$

Next, we show
\[
\bigtriangleup_{g}u\geq0.
\]
\ When $D^{2}u$ is diagonalized,
\[
\bigtriangleup_{g}u=\sum_{i=1}^{n}g^{ii}u_{ii}=\sum_{i=1}^{n}\frac{\lambda
_{i}}{1+\lambda_{i}^{2}}=\frac{1}{2}\sum_{i=1}^{n}\sin(2\theta_{i}).
\]
Let $S\subset{\mathbb{R}}^{n}$ be the hypersurface (with boundary) given by
\[
S=\left\{  \theta\ |\ \theta_{1}+\theta_{2}+\cdots+\theta_{n}=\frac{\pi}%
{2}(n-2),|\theta_{i}|\leq\frac{\pi}{2}\right\}  \,,
\]
where $\theta=\left(  \theta_{1},\cdots,\theta_{n}\right)  .$ Set
$\Gamma(\theta)=\frac{1}{2}\sum_{i=1}^{n}\sin(2\theta_{i}).$ Suppose that
$\Gamma$ obtains a negative minimum on the interior of $S$ at $\theta^{\ast}.$
At this point $D\Gamma$ vanishes on $T_{\theta^{\ast}}S,$ thus we have%
\[
\cos(2\theta_{i})=\cos\left(  2\theta_{j}\right)  ,\ \text{then}\ \theta
_{i}=\pm\theta_{j}.
\]
The only two possible configurations for $\theta$ are%
\begin{align*}
\theta_{1}  &  =\cdots=\theta_{n}=\frac{\left(  n-2\right)  \pi}%
{2n}\ \ \text{or}\\
\theta_{1}  &  =\cdots=\theta_{n-2}=\frac{\pi}{2},\ \theta_{n-1}=-\theta_{n}.
\end{align*}
In either case, $\Gamma$ is nonnegative. This contradiction allows us to
verify the nonnegativity of $\Gamma$ along the boundary $\partial S.$ It
follows easily that $\Gamma\geq0$ there by induction on dimension $n,$ as
\[
\partial S=%
%TCIMACRO{\dbigcup \limits_{k=1}^{n}}%
%BeginExpansion
{\displaystyle\bigcup\limits_{k=1}^{n}}
%EndExpansion
\left\{  \theta\ |\ \theta_{1}+\cdots+\widehat{\theta_{k}}+\cdots+\theta
_{n}=\frac{\pi}{2}(n-3),\ |\theta_{i}|\leq\frac{\pi}{2}\right\}  \,
\]
and \ $\Gamma\left(  \theta_{1},\theta_{2}\right)  =0$ for $\theta_{1}%
+\theta_{2}=0.$

\ Further, we show%
\[
\bigtriangleup_{g}\left\vert Du\right\vert \geq0.
\]
We calculate%
\begin{align*}
\bigtriangleup_{g}\left\vert Du\right\vert  &  =\sum_{\alpha,\beta=1}%
^{n}g^{\alpha\beta}\partial_{\alpha\beta}\left\vert Du\right\vert
=\sum_{\alpha,\beta,i=1}^{n}g^{\alpha\beta}\left(  \frac{u_{i}u_{i\beta\alpha
}}{\left\vert Du\right\vert }+\frac{u_{i\alpha}u_{i\beta}}{\left\vert
Du\right\vert }-\sum_{j=1}^{n}\frac{u_{i}u_{i\beta}u_{j}u_{j\alpha}%
}{\left\vert Du\right\vert ^{3}}\right) \\
&  =\sum_{\alpha,\beta,i=1}^{n}g^{\alpha\beta}\left(  \frac{u_{i\alpha
}u_{i\beta}}{\left\vert Du\right\vert }-\sum_{j=1}^{n}\frac{u_{i}u_{i\beta
}u_{j}u_{j\alpha}}{\left\vert Du\right\vert ^{3}}\right) \\
&  \overset{D^{2}u\ \text{{\small is\ diagonal}}}{=}\sum_{\alpha=1}%
^{n}g^{\alpha\alpha}\frac{\left(  \left\vert Du\right\vert ^{2}-u_{\alpha}%
^{2}\right)  \lambda_{\alpha}^{2}}{\left\vert Du\right\vert ^{3}}\geq0,
\end{align*}
where we used the minimality equation (\ref{Emin}).

\ Combining the subharmonicity of $u$ and $\left\vert Du\right\vert $ with
(\ref{Inequ g^nn}) and (\ref{E critical gradient}), we have at $x^{\ast}$%
\begin{align*}
0  &  \geq\bigtriangleup_{g}w=\left\vert Du\right\vert \bigtriangleup_{g}%
\eta+2\sum_{\alpha=1}^{n}g^{\alpha\alpha}\eta_{\alpha}\left\vert Du\right\vert
_{\alpha}+\underset{\geq0}{\underbrace{\eta\bigtriangleup_{g}\left\vert
Du\right\vert +2Au\bigtriangleup_{g}u}}+2A\sum_{\alpha=1}^{n}g^{\alpha\alpha
}u_{\alpha}^{2}\\
&  \geq\left\vert Du\right\vert \bigtriangleup_{g}\eta+2\sum_{\alpha=1}%
^{n}g^{\alpha\alpha}\eta_{\alpha}\underline{\left\vert Du\right\vert _{\alpha
}}+2A\sum_{\alpha=1}^{n}g^{\alpha\alpha}u_{\alpha}^{2}\\
&  \geq-2ng^{nn}\left\vert Du\right\vert -2\sum_{\alpha=1}^{n}g^{\alpha\alpha
}\eta_{\alpha}\left(  \frac{\eta_{\alpha}\left\vert Du\right\vert
+2Auu_{\alpha}}{\eta}\right)  +\frac{2A}{n}g^{nn}\left\vert Du\right\vert
^{2}\\
&  \geq-2ng^{nn}\left\vert Du\right\vert -4g^{nn}\frac{\left\vert
Du\right\vert }{\eta}-8g^{nn}Au\frac{\left\vert Du\right\vert }{\eta}%
+\frac{2A}{n}g^{nn}\left\vert Du\right\vert ^{2};
\end{align*}
It follows that
\[
0\geq-2n\eta-4-8Au+\frac{2A}{n}\eta\left\vert Du\right\vert .
\]
Then by the assumption $M\leq u\leq2M$ and $A=n/M$
\[
\eta\left\vert Du\right\vert \left(  x^{\ast}\right)  \leq\left(
n+2+8n\right)  M.
\]
So we obtain
\begin{equation}
\left\vert Du\left(  0\right)  \right\vert \leq w\left(  x^{\ast}\right)
\leq15nM. \label{gradient estimate critical phase}%
\end{equation}

\textbf{Case }$\Theta>\left(  n-2\right)  \pi/2.$ Let $\Theta=\delta+\left(
n-2\right)  \pi/2.$ From our special Lagrangian equation (\ref{EsLag}), we
know
\[
\theta_{i}+\left(  n-1\right)  \frac{\pi}{2}>\left(  n-2\right)  \frac{\pi}%
{2}+\delta\ \ \text{or \ }\theta_{i}>-\frac{\pi}{2}+\delta.
\]
We can control the gradient of the convex function $u\left(  x\right)
+\frac{1}{2}\max\left\{  \cot\delta,0\right\}  x^{2}$ by its oscillation,
thus
\begin{equation}
\left\vert Du\left(  0\right)  \right\vert \leq\operatorname*{osc}_{B_{1}%
}+\frac{1}{2}\max\left\{  \cot\delta,0\right\}  . \label{Rough bound}%
\end{equation}
In order to get rid of the $\delta$-dependence in the gradient estimate, we
need the following.

\begin{proposition}
Let smooth $u$ satisfy (\ref{EsLag}) with $\Theta-\left(  n-1\right)
\frac{\pi}{2}=\delta\in(0,\pi/4)$ on $B_{2}(0).$ Suppose that%
\begin{equation}
\operatorname*{osc}_{B_{2}}u\leq\frac{1}{2\sin\delta}. \label{osc condition}%
\end{equation}
Then
\[
\left\vert Du(0)\right\vert \leq C(n)\left(  \operatorname*{osc}_{B_{2}%
}u+1\right)  .
\]

\end{proposition}

\begin{proof}
We take the Lewy rotation in the proof of Proposition \ref{Prop rotation}, to
obtain a \textquotedblleft critical\textquotedblright\ representation
$\mathfrak{M}=\left(  \bar{x},D\bar{u}\left(  \bar{x}\right)  \right)  $ for
the original special Lagrangian graph $\mathfrak{M}=(x,Du(x))$ with $x\in
B_{2}.$ Recentering the new coordinates, we take
\begin{equation}
\left\{
\begin{array}
[c]{l}%
\bar{x}=x\cos\frac{\delta}{n}+Du\left(  x\right)  \sin\frac{\delta}%
{n}-Du\left(  0\right)  \sin\frac{\delta}{n}\\
D\bar{u}\left(  \bar{x}\right)  =-x\sin\frac{\delta}{n}+Du\left(  x\right)
\cos\frac{\delta}{n}%
\end{array}
\right.  . \label{critical'}%
\end{equation}
By (\ref{lbonrad}) we see that the potential $\bar{u}$ is defined on a ball in
$\bar{x}$-space around the origin of radius
\[
\bar{R}=\frac{2}{2\cos(\frac{\delta}{n})}>1.
\]
From (\ref{critical'}) and the estimate
(\ref{gradient estimate critical phase}) for the critical potential, we have%
\[
\left\vert Du(0)\right\vert =\frac{\left\vert D\bar{u}(\bar{0})\right\vert
}{\cos(\delta/n)}\leq C\left(  n\right)  \operatorname*{osc}_{\bar{B}_{1}}%
\bar{u}.
\]

Next,we estimate the oscillation of $\bar{u}$ in terms of $u.$ We may assume
that $\bar{u}(\bar{0})=0.$ Without loss of generality we assume the maximum of
$|\bar{u}|$ on $\bar{B}_{1}(\bar{0})$ happens along the positive $\bar{x}_{1}%
$-axis, and even on the boundary $\partial\bar{B}_{1}.$ Thus we have
\[
\operatorname*{osc}_{\bar{B}_{1}}\bar{u}\leq2\left\vert \int_{\bar{x}_{1}%
=0}^{\bar{x}_{1}=1}\bar{u}_{\bar{x}_{1}}d\bar{x}_{1}\right\vert .
\]
In the following, we convert the integral of $\bar{u}_{\bar{x}_{1}}$ to one in
terms of $u_{x_{1}},$ then recover the oscillation of $\bar{u}$ from that of
$u.$

We work on the $x_{1}$-$y_{1}$ plane in the remaining of the proof. Under our
above assumption, the $\bar{x}_{1}$-axis is given by the line
\[
y_{1}=\tan\left(  \frac{\delta}{n}\right)  x_{1}%
\]
and the curve $\gamma:(x_{1},u_{1}(x_{1}))$ with $\left\vert x_{1}\right\vert
\,<2$ forms a graph over the $\bar{x}_{1}$-axis. \ Let $l_{0}$ be the line
perpendicular to the $\bar{x}_{1}$-axis and intersecting the curve $\gamma$ at
$\left(  0,u_{1}\left(  0\right)  \right)  $ along the $y_{1}$-axis. The
intersection of $l_{0}$ and the $\bar{x}_{1}$-axis (which is also the origin
of the recentered the $\bar{x}_{1}$-$\bar{y}_{1}$ plane) has distance to the
origin of the $x_{1}$-$y_{1}$ plane given by
\begin{equation}
\left\vert u_{1}\left(  0\right)  \right\vert \sin\left(  \frac{\delta}%
{n}\right)  \leq\left(  \operatorname*{osc}_{B_{1}}u+\frac{1}{2}\cot
\delta\right)  \sin\left(  \frac{\delta}{n}\right)  \leq1 \label{kick}%
\end{equation}
by the rough bound (\ref{Rough bound}) and the condition (\ref{osc condition}%
). Now let $l_{1}$ be the line parallel to $l_{0}$ passing through the point
$\bar{x}_{1}=1$ along the $\bar{x}_{1}$-axis.

The integral
\[
\int_{\bar{x}_{1}=0}^{\bar{x}_{1}=1}\bar{u}_{\bar{x}_{1}}d\bar{x}_{1}%
\]
is the signed area between the $\bar{x}_{1}$-axis and the curve $\gamma,$ and
lying between the lines $l_{0}$ and $l_{1}.$ We convert this to an integral
over $x_{1},$
\[
\int_{\bar{x}_{1}=0}^{\bar{x}_{1}=1}\bar{u}_{\bar{x}_{1}}d\bar{x}_{1}%
=\int_{P(l_{0}\cap\bar{x}_{1}\text{-axis})}^{P(l_{1}\cap\bar{x}_{1}%
\text{-axis})}\left[  u_{1}(x_{1})-\tan\left(  \frac{\delta}{n}\right)
x_{1}\right]  dx_{1}+K_{0}+K_{1},
\]
where $P$ denotes projection to the $x_{1}$-axis, and $K_{0}\ $as well as
$K_{1}$ denotes the signed areas to the left or right of the desired region,
forming the difference.

It is important to note the following for $j=1,2:$

(i)$\ P(l_{j}\cap\bar{x}_{1}$-axis$)\ $is in the $x_{1}$-domain of $u_{1}$ by
(\ref{kick}),%
\begin{align*}
\left\vert P(l_{0}\cap\bar{x}_{1}\text{-axis})\right\vert  &  \leq1\cdot
\cos\left(  \frac{\delta}{n}\right)  <1,\\
\left\vert P(l_{1}\cap\bar{x}_{1}\text{-axis})\right\vert  &  \leq\left(
1+1\right)  \cdot\cos\left(  \frac{\delta}{n}\right)  <2;
\end{align*}
(ii) $P(l_{j}\cap\gamma)\ $is also in the $x_{1}$-domain of $u_{1}$ as the
whole Lagrangian surface $\mathfrak{M}$ is a graph over $B_{2},$%
\[
\left\vert P(l_{j}\cap\gamma)\right\vert \leq2;
\]
(iii) the region $K_{j}$ is bounded by the line $l_{j},$ the vertical line
$x_{1}=P(l_{j}\cap\bar{x}_{1}$-axis$),$ and the curve $\gamma,$ also each
region $K_{j}$ is on one side of the $\bar{x}_{1}$-axis.

Thus from (i)
\[
\left\vert \int_{P(l_{0}\cap\bar{x}_{1}\text{-axis})}^{P(l_{1}\cap\bar{x}%
_{1}\text{-axis})}\left[  u_{1}(x_{1})-\tan\left(  \frac{\delta}{n}\right)
x_{1}\right]  dx_{1}\right\vert \leq\operatorname*{osc}_{B_{2}}u+C(n)
\]
and from (ii) (iii)%
\[
|K_{j}|\leq\left\vert \int_{P(l_{j}\cap\bar{x}_{1}\text{-axis})}^{P\left[
l_{j}\cap\gamma\right]  }\left[  u_{1}(x_{1})-\tan(\frac{\delta}{n}%
)x_{1}\right]  dx_{1}\right\vert \leq\operatorname*{osc}_{B_{2}}u+C(n).
\]
It follows that we have the conclusion of Proposition 5.1
\[
\left\vert Du(0)\right\vert \leq C\left(  n\right)  \operatorname*{osc}%
_{\bar{B}_{1}}\bar{u}\leq C\left(  n\right)  \left(  \operatorname*{osc}%
_{B_{2}}u+1\right)  .
\]

\end{proof}

We finish the proof of Theorem 1.3. \ For $\delta\geq\pi/4,$ the bound
(\ref{Rough bound}) gives
\[
\left\vert Du\left(  0\right)  \right\vert \leq\operatorname*{osc}_{B_{1}%
}u+\frac{1}{2}\leq C(n)\left[  \operatorname*{osc}_{B_{2}}u+1\right]  .
\]
For $\delta\leq\pi/4,$ if $\operatorname*{osc}_{B_{2}}u\leq1/\left(
2\sin\delta\right)  ,$ then Proposition 4.1 gives
\[
\left\vert Du(0)\right\vert \leq C(n)\left[  \operatorname*{osc}_{B_{2}%
}u+1\right]  .
\]
Otherwise, $\ \operatorname*{osc}_{B_{2}}u>1/\left(  2\sin\delta\right)  ,$
and from (\ref{Rough bound})
\[
\left\vert Du\left(  0\right)  \right\vert \leq\operatorname*{osc}_{B_{1}%
}u+\operatorname*{osc}_{B_{2}}u\leq C(n)\left[  \operatorname*{osc}_{B_{2}%
}u+1\right]  .
\]
Applying this estimate on $B_{2}(x)$ for any $x\in B_{1}(0),$ we arrive at the
conclusion of Theorem 1.3.

\end{document}